\documentclass{amsart}
\usepackage[T1]{fontenc}
 \usepackage[utf8]{inputenc}

\usepackage{amsmath}
\usepackage{latexsym,amsfonts,amssymb,mathrsfs}
\usepackage{amsthm}
\usepackage{mathdots}
\usepackage{color}
\usepackage{enumitem} 
\usepackage{xspace} 
 \usepackage{array}
 \usepackage{multicol} 
 \usepackage{mathtools} 
 \usepackage[pagebackref]{hyperref}
\usepackage[capitalise]{cleveref}

\usepackage{tikz}
\usetikzlibrary{cd}
\usetikzlibrary{calc}
\usetikzlibrary{math}
\usetikzlibrary{arrows}
\usetikzlibrary{fit}
\usetikzlibrary{positioning}
\usetikzlibrary{decorations.pathmorphing}
\usetikzlibrary{decorations.pathreplacing}
\usetikzlibrary{ decorations.markings} 
\tikzset{Rightarrow/.style={double equal sign distance,>={Implies},->},
triple/.style={-,preaction={draw,Rightarrow}},
quadruple/.style={preaction={draw,Rightarrow,shorten >=0pt},shorten >=1pt,-,double,double
distance=0.2pt}}

\tikzset{%
    symbol/.style={%
        draw=none,
        every to/.append style={%
            edge node={node [sloped, allow upside down, auto=false]{$#1$}}}
    }
}

\tikzset{%
scalearrow/.style n args={3}{
  decoration={
    markings,
    mark=at position (1-#1)/2*\pgfdecoratedpathlength
      with {\coordinate (#2);},
    mark=at position (1+#1)/2*\pgfdecoratedpathlength
      with {\coordinate (#3);},
    },
  postaction=decorate,
  }
}

\theoremstyle{plain}   
\newtheorem{thm}{Theorem}[section] 
\makeatletter\let\c@thm\c@thm\makeatother

\makeatletter\let\c@cor\c@thm\makeatother
\newtheorem{lem}{Lemma}[section]
\makeatletter\let\c@lem\c@thm\makeatother
\newtheorem{prop}{Proposition}[section]
\makeatletter\let\c@prop\c@thm\makeatother

\makeatletter\let\c@claim\c@thm\makeatother

\makeatletter\let\c@conjecture\c@thm\makeatother

\makeatletter\let\c@wconjecture\c@thm\makeatother

\newtheorem*{unnumberedtheorem}{Theorem}
\newtheorem*{unnumberedcorollary}{Corollary}

\theoremstyle{definition}

\newtheorem{defn}{Definition}[section]
\makeatletter\let\c@defn\c@thm\makeatother
\newtheorem{const}{Construction}[section]
\makeatletter\let\c@const\c@thm\makeatother
\newtheorem{notn}{Notation}[section]
\makeatletter\let\c@notn\c@thm\makeatother

\makeatletter\let\c@convention\c@thm\makeatother

\makeatletter\let\c@convention\c@thm\makeatother

\theoremstyle{remark}

\newtheorem{rmk}{Remark}[section]
\makeatletter\let\c@rmk\c@thm\makeatother

\makeatletter\let\c@ex\c@thm\makeatother

\makeatletter\let\c@observation\c@thm\makeatother

\makeatletter\let\c@warning\c@thm\makeatother

\makeatletter\let\c@digression\c@thm\makeatother

\makeatletter\let\c@answ\c@thm\makeatother

\makeatletter\let\c@answ\c@thm\makeatother

\makeatletter\let\c@aside\c@thm\makeatother

\makeatletter
\let\c@equation\c@thm
\numberwithin{equation}{section}
\makeatother

\crefname{lem}{Lemma}{Lemmas}
\crefname{thm}{Theorem}{Theorems}
\crefname{defn}{Definition}{Definitions}
\crefname{notn}{Notation}{Notations}
\crefname{const}{Construction}{Constructions}
\crefname{prop}{Proposition}{Propositions}
\crefname{rmk}{Remark}{Remarks}
\crefname{cor}{Corollary}{Corollaries}
\crefname{equation}{Display}{Displays}
\crefname{ex}{Example}{Examples}
\crefname{thmalph}{Theorem}{Theorems}
\crefname{answ}{Answer}{Answers}
\crefname{question}{Question}{Questions}

\makeatletter
\newcommand{\@bbify}[1]{
  \ifcsname b#1\endcsname
  \message{WARNING: Overwriting b#1 with blackboard letter!}
  \fi
  \expandafter\edef\csname b#1\endcsname
  {\noexpand\ensuremath{\noexpand\mathbb #1}\noexpand\xspace}}
\newcommand{\@calify}[1]{
  \ifcsname c#1\endcsname
  \message{WARNING: Overwriting c#1 with calligraphic letter!}
  \fi 
  \expandafter\edef\csname c#1\endcsname
  {\noexpand\ensuremath{\noexpand\mathcal #1}\noexpand\xspace}}
\newcommand{\@bfify}[1]{
  \ifcsname bf#1\endcsname
  \message{WARNING: Overwriting c#1 with bold letter!}
  \fi
  \expandafter\edef\csname bf#1\endcsname
  {\noexpand\ensuremath{\noexpand\mathbf #1}\noexpand\xspace}}
\newcounter{@letter}\stepcounter{@letter}
\loop\@bbify{\Alph{@letter}}\@calify{\Alph{@letter}}\@bfify{\Alph{@letter}}
\ifnum\the@letter<26\stepcounter{@letter}\repeat
\makeatother

\newcommand{\cat}{\cC\!\mathit{at}}

 \newcommand{\omegacat}{\omega\cat}

\DeclareMathOperator{\Ob}{Ob}

\DeclareMathOperator*{\colim}{colim}
\DeclareMathOperator{\id}{id}
\DeclareMathOperator{\mor}{mor}
\DeclareMathOperator{\eq}{eq}
\DeclareMathOperator{\bieq}{bieq}
\newcommand{\inttrunc}[1]{\tau_{\leq #1}^{\mathrm{i}}}

\newcommand{\indind}{k}
\newcommand{\firstiso}{\alpha}
\newcommand{\secondiso}{\beta}

\newcommand\omegat{\texorpdfstring{$\omega$}{omega}}
\newcommand\nunderset[2]{\;\underset{\mathclap{#1}}{#2}\;}
\AtBeginDocument{%
   \def\MR#1{}
}

\author[Hadzihasanovic]{Amar Hadzihasanovic}
\address{Tallinn University of Technology, 
Tallinn,
Estonia
}
\email{amar.hadzihasanovic@taltech.ee}

\author[Loubaton]{F\'elix Loubaton}
\address{Max Planck Institute for Mathematics, Bonn, Germany}
\email{loubaton@mpim-bonn.mpg.de}

\author[Ozornova]{Viktoriya Ozornova}
\address{Max Planck Institute for Mathematics, Bonn, Germany}
\email{viktoriya.ozornova@mpim-bonn.mpg.de}

\author[Rovelli]{Martina Rovelli}
\address{
University of Massachusetts Amherst, 
Amherst (MA),
USA
}
\email{mrovelli@umass.edu} 


\title{A model for the coherent walking \omegat-equivalence}

\begin{document}

\maketitle

\begin{abstract}
We prove that a certain $\omega$-category, which was constructed in previous work by the third and fourth author, is a model for the fully coherent walking $\omega$-equivalence. Further, appropriate truncations of it give models for the fully coherent walking $n$-equivalence for each $n\geq1$.
\end{abstract}

\section*{Introduction}

An \emph{$\omega$-category} is a type of strict categorical structure which allows for cells in each positive dimension, together with composition and identity operators, which satisfy strict axioms of associativity, unitality and interchange. When all cells are identities past dimension $n$ one refers to an \emph{$n$-category}, recovering the well known instances of a set, category and $2$-category, when $n=0,1,2$. Both $n$-categories for $n\geq0$ and $\omega$-categories are prominent in the literature, and are studied e.g.~in \cite{StreetOrientedSimplexes,SteinerOmegaCatChain,LMW,AraMaltsiniotisJoin}.

Given the strictness of the axioms, examples that occur naturally in mathematical nature (such as various higher categories of cobordisms and spans, and higher Morita categories) do not generally assemble into a strict $\omega$- or $n$-category. These generally form a weak infinite-dimensional category, often referred to as an $(\infty,\infty)$-category, given that the definition of composition operators is only weakly well-defined and the axioms only hold weakly. Nevertheless, developing an understanding for strict $\omega$- and $n$-categories is crucial to tackle the study of weak $(\infty,\infty)$- and $(\infty,n)$-categories:
\begin{itemize}[leftmargin=*, label={>>}]
    \item Strict $\omega$-categories often parameterize operations and interesting quantities in weak higher categories.
    This is the approach taken, e.g., in \cite{rezkTheta,RiehlVerityMonads,HORR,FHM} where strict $\omega$- or $n$-categories are used to parameterize free composites, (homotopy coherent) adjunctions, and pasting diagrams in a weak higher category.
    \item Strict $\omega$-categories provide a first --- and yet non-trivial --- approximation of the theory of weak $(\infty,\infty)$-categories (cf.~\cite{VerityComplicialI,GoldthorpeWeakEnrichment, GoldthorpeSheaves}), and as such they can be used as a playground to better understand the behavior of infinite-dimensional higher categories.
    \item In the theory of polygraphs, strict $\omega$-categories model higher-dimensional rewrite systems, such as those arising from presentations by generators and relations of groups, monoids, and higher algebraic structures (cf.~\cite{PolyBook}).
\end{itemize}

For reasons discussed in \cite{ORSurvey}, it is necessary
to understand which is the intrinsic notion of sameness for two objects inside a given $n$- or $\omega$-category. This is typically expressed by requiring the existence of a $1$-cell between said objects, together with other cells witnessing that the $1$-cell is ``reversible'' in a suitably weak sense. Properties of such notion of sameness, which we refer to as \emph{$\omega$-equivalence} or \emph{$n$-equivalence}, have been studied e.g.~in \cite{ChengOmegaDuals,GurskiBieq,AraLucas,HadzihasanovicDiagrammatic,RiceCoinductive,clingmanThesis,FHM,HenryLoubaton,LoubatonNerves,ORSurvey}.

One can formally identify an $\omega$-category $\omega\cE$ (resp.~$n$-category $(n-1)\cE$) that classifies $\omega$-equivalences (resp.~$(n-1)$-equivalences); cf.~e.g.~\cite[Remark 4.4]{AraLucas}). For instance, for $n=1,2$ we would get, respectively, the walking isomorphism $\cI$ and the walking equivalence $\cE$ considered e.g.~in \cite{lack2}. However, these known candidates are known to lack coherence as soon as $n\geq 2$. 
More precisely, $\omega\cE$ (resp.~$(n-1)\cE$ for $n\geq 2$) is known to not be contractible in the model structure on the category of $\omega$-categories (resp.~$n$-categories) from \cite{LMW}.

As showcased, for instance, in \cite{lack1,lack2,ORNerves2Cat} for the case $n=2$, it is important to have at one's disposal contractible models of the fully coherent $(n-1)$-equivalence. 
When checking whether a $1$-morphism inside a $2$-category is an equivalence, it is sufficient to look at incoherent equivalences. However, if one wants the data witnessing an equivalence to be essentially unique, then this is encoded by a coherent equivalence, as discussed in \cite{lack2}. This principle was also used in \cite{ORNerves2Cat} to enhance the Duskin nerve to a right Quillen functor from $2$-categories to multiply marked simplicial sets. Indeed, an explicit walking coherent equivalence allows for an explicit construction for the localization of a $2$-, or more generally $n$- or even $\omega$-category at a set of cells, by attaching the walking coherent equivalence at each of those cells.

For all $n>0$, we know for abstract reasons (cf.~\cite[\textsection4.7]{LMW})
that there must exist a contractible $\omega$-category (resp.~$n$-category) with two objects, and it is shown in \cite[Proposition~20.4.5]{PolyBook}
that such $\omega$-category (resp.~$n$-category) will automatically classify $\omega$-equivalences (resp.~$(n-1)$-equivalences). Hence, one such $\omega$-category (resp.~$n$-category) deserves to be referred to as a \emph{fully coherent walking $\omega$-equivalence}
(resp.~\emph{fully coherent walking $(n-1)$-equivalence}).
For $n=1$, one can take as a model for the coherent $(n-1)$-equivalence again $\cI$, the usual walking isomorphism, and for $n=2$ one can take $\cE^{\mathrm{adj}}$, the walking adjoint equivalence. For $n=3$, it is likely --- yet unknown --- that the $3$-category $\mathrm{bi}\cE^{\mathrm{adj}}$ (cf. \cite[\textsection2]{GurskiBieq}) is a model for the fully coherent walking $2$-equivalence. No model for the fully coherent $n$-equivalence for $n>2$ or $\omega$-equivalence is known.

In \cite[\textsection1.5]{ORSurvey}, a candidate $\widehat{\omega\cE}$ for the fully coherent walking $\omega$-equivalence, which is a polygraph and of finite type, was introduced by the third- and fourth-named authors. 
Using the theory of marked $\omega$-categories from \cite{HenryLoubaton}, we show in this paper as \cref{omegaEcontractible} that $\widehat{\omega\cE}$ is a contractible $\omega$-category. In particular, it indeed realizes the fully coherent walking $\omega$-equivalence.

\begin{unnumberedtheorem}
The possibly coherent walking $\omega$-equivalence $\widehat{\omega\cE}$ from \cite[Construction~1.5.13]{ORSurvey} is indeed a model for the coherent walking $\omega$-equivalence.
\end{unnumberedtheorem}

The intelligent $n$-truncation functor $\inttrunc{n}\colon\omega\cat\to n\cat$
from \cite[\textsection 1.2]{AraMaltsiniotisJoin}
is shown in \cite[\textsection 6]{LMW} to be a left Quillen functor, and as such it preserves categorical equivalences between polygraphs. In particular, as a consequence of the theorem we also obtain that $\inttrunc{n}\widehat{\omega\cE}$ is a contractible $n$-category of finite type, so it realizes the fully coherent walking $(n-1)$-equivalence, for each $n>0$:

\begin{unnumberedcorollary}
Given $n>0$, the intelligent truncation $\inttrunc{n}\widehat{\omega\cE}$ of the possibly coherent walking $\omega$-equivalence is a model for the coherent walking $n$-equivalence.
\end{unnumberedcorollary}

\addtocontents{toc}{\protect\setcounter{tocdepth}{1}}
\subsection*{Acknowledgements}

The content of this note benefited from conversations with 
Dimitri Ara, Lennart Meier, Fran\c cois M\'etayer, Samuel Mimram,
and Alex Rice.
This work started during a visit of the first-named author to MPIM Bonn, supported by the Estonian Research Council grant PSG764, and it was completed during a wonderful conference at the University of Utrecht, sponsored by NWO OCENW.KLEIN.364.
The fourth-named author is grateful for support from the National Science Foundation under Grant No. DMS-2203915.

\addtocontents{toc}{\protect\setcounter{tocdepth}{2}}

\section{The model for the coherent \omegat-equivalence}

\subsection{\omegat-categories}

We refer the reader to e.g.~ \cite[\textsection 3.2]{LMW}
for the notion of an \emph{$\omega$-category} and \emph{$\omega$-functor}. Roughly speaking, an \emph{$\omega$-category} $\cD$ consists of a set of \emph{$n$-cells} $\cD_n$ for $n\geq0$, together with domain and codomain operators
$d^+,d^-\colon\cD_n\to\cD_{n-k}$,
composition operators
$*_{n-k}\colon \cD_n\times_{\cD_{n-k}}\cD_n\to\cD_n$,
and identity operators
$\id\colon\cD_{n-k}\to\cD_{n}$ for $0<k\leq n$,
satisfying strictly appropriate associativity, unitality, and interchange axioms. We follow the convention that $g*_{n-k}f$ is defined whenever $d_{n-k}^+f=d_{n-k}^-g$.
An \emph{$\omega$-functor} $F\colon\cD\to\cE$
consists of an assignment $F_n\colon\cD_n\to\cE_n$ that commutes with all relevant operators.

We collect here the $\omega$-categories and constructions of such that will play a role in this paper.
\begin{itemize}[leftmargin=*, label={>>}]
    \item Given $n\geq0$, we denote by $\cC_n$ the \emph{walking $n$-cell}, a.k.a.~ \emph{$n$-disk} and \emph{$n$-globe},
    which is freely generated by an $n$-cell, and we denote by $\partial\cC_n$ its boundary, which is freely generated by two $(n-1)$-cells which have the same domain and codomain. 
    \item Given an $\omega$-category $\cD$, we denote by $\cD^{\circ}$ the \emph{total dual $\omega$-category} of $\cD$, which, roughly speaking, has the same sets of $n$-cells but swaps the domain and codomain operators. This construction is considered e.g.~ in \cite[\textsection 1.8]{AraMaltsiniotisJoin}. 
    \item Given two $\omega$-categories $\cA$ and $\cB$, we denote by $\cA\amalg\cB$ the \emph{disjoint union} of $\cA$ and $\cB$, which is defined as the categorical coproduct in $\omega\cat$ and has the disjoint union of the sets of $n$-cells of $\cA$ and $\cB$ as the set of $n$-cells.
    \item The category $\omega\cat$ is cocomplete (see e.g.~\cite[Corollary 14.2.5]{PolyBook}), and 
    we denote by $\mathrm{colim}_{i\in\cI}\cD_i$ the \emph{colimit} in $\omega\cat$ of a diagram $i\in\cI\mapsto\cD_i$. 
    \item Given an $\omega$-category $\cD$, we denote by $\Sigma\cD$ the \emph{suspension} of $\cD$, which is freely generated by two objects and one $(n+1)$-cell $\Sigma a$ between them for each $n$-cell $a$ of $\cD$. A version of this construction is considered e.g.~ in \cite[\textsection2.2]{ORquillen}.
    \item Given an $\omega$-category $\cD$ and two objects $a$ and $b$, we denote by $\hom_{\cD}(a,b)$ the \emph{hom-$\omega$-category} of $\cD$ from $a$ to $b$, which has one $n$-cell for every $(n+1)$-cell $f$ of $\cD$ for which $d_0^+f=b$ and $d_0^-f=a$.
\end{itemize}

The existence of the following adjunction can be checked by direct inspection (cf. \cite[\textsection B.6.5]{AraMaltsiniotisJoin}). The preservation of connected colimits can be deduced using a standard argument based on \cite[Proposition 2.9]{HirschhornOvercategories}.

\begin{prop}
\label{AdjointOfSuspension}
If $\omega\cat_{*,*}$ denotes the category of bipointed marked $\infty$-categories and bipointed $\omega$-functors,
there is an adjunction
\[\Sigma\colon\omega\cat\rightleftarrows\omega\cat_{*,*} :\!\hom\]
Moreover, the functor $\Sigma\colon\omega\cat\to\omega\cat$ preserves connected colimits.
\end{prop}

\subsection{Equivalences and bi-equivalences in an \omegat-category}

The following is originally due to M\'etayer, and is also considered in \cite[\textsection1.2]{AraLucas} (under the terminology of \emph{structure of reversibility}) and \cite[D\'efinition 1.1.7]{LoubatonNerves} (under the terminology of \emph{ensemble d'inversibilit\'e}).

\begin{defn}
\label{InvertibilitySet}
Let $\cD$ be an $\omega$-category. An \emph{invertibility set}
in $\cD$ is a set $E= \coprod_{n>0}E_n$ with $E_n\subseteq\cD_n$ 
such that, 
for all $n>0$ and $a\in E_n$, there exists $\tilde{a}\in E_n$
of the form
\[
\tilde a\colon d^+_{n-1}a\to d^-_{n-1}a
\]
and $c,c'\in E_{n+1}$ of the form
\[
c \colon \tilde{a}\underset{n-1}*a \to  \id_{d^-_{n-1} a} \quad\text{ and }\quad
c'\colon   a\underset{n-1}*\tilde{a} \to \id_{d^+_{n-1} a}.
\]
In the situation above we say that $\tilde a$ is a \emph{weak inverse} for $a$.
\end{defn}

\begin{defn}
\label{Equivalence}
Let $\cD$ be an $\omega$-category and $n>0$. Given $a\in\cD_n$, the $n$-cell $a$ is said to be an \emph{$\omega$-equivalence}
if there exists an invertibility set $E$ such that $a\in E$.
We denote by $\eq_n\cD$ the set of
all $n$-cells in $\cD$ that are $\omega$-equivalences
and by $\eq\cD\coloneqq\coprod_{n>0}\eq_{n}\cD$ the set of all $\omega$-equivalences in $\cD$.
\end{defn}

The following is from \cite[\textsection1.2]{AraLucas} and \cite[Lemme~1.1.8]{LoubatonNerves},
and is generally taken as the defining property for the set $\eq\cD$ of $\omega$-equivalences in an $\omega$-category $\cD$ (see e.g.~\cite[Definition~6]{LMW}).

\begin{prop}\label{CoinductiveDefinitionEquivalences}
Let $\cD$ be an $\omega$-category and $n>0$. Given $a\in\cD_n$, we have that $a\in\eq_n\cD$ if and only if there exist $\tilde a\in\cD_n$
of the form
\[
\tilde a:d^+_{n-1}a\to d^-_{n-1}a
\]
and $c,c'\in\eq_{n+1}\cD$ of the form
\[
c \colon \tilde{a}\underset{n-1}*a \to  \id_{d^-_{n-1} a} \quad\text{ and }\quad
c'\colon   a\underset{n-1}*\tilde{a} \to \id_{d^+_{n-1} a}.
\]
\end{prop}

\begin{rmk}
\label{InvertibilitySetOfEquivalences}
Given an $\omega$-category $\cD$, by \cref{CoinductiveDefinitionEquivalences} the set $\eq\cD$ is the maximal invertibility set in $\cD$.
\end{rmk}

\begin{defn}
\label{BiInvertibilitySet}
Let $\cD$ be an $\omega$-category.
A \emph{bi-invertibility set} in $\cD$ is a set $E= \coprod_{n>0}E_n$ with $E_n\subseteq\cD_n$ 
such that,
for all $n>0$ and $a\in E_n$, there exist $ a^L, a^R\in\cD_{n}$ of the form
\[ a^L, a^R\colon d^+_{n-1}a\to d^-_{n-1}a\]
and $c,c'\in E_{n+1}$ of the form
\[c \colon   a^L\underset{n-1}*a\to \id_{d^-_{n-1} a} \quad\text{ and }\quad
c' \colon   a\underset{n-1}* a^R \to \id_{d^+_{n-1} a}.\]
In the situation above, we say that $a^L$, resp. $a^R$, is a \emph{left inverse}, resp. \emph{right inverse}, for $a$.
\end{defn}

\begin{defn}
\label{BiEquivalence}
Given an $\omega$-category $\cD$ and $a\in\cD_n$ with $n>0$, the $n$-cell $a$ is said to be an \emph{$\omega$-bi-equivalence}
if there exists a bi-invertibility set $E$ such that $a\in E$.
We denote by $\bieq_n\cD$ the set of
all $n$-cells in $\cD$ that are $\omega$-bi-equivalences
and by $\bieq\cD\coloneqq\coprod_{n>0}\bieq_n\cD$ the set of all $\omega$-bi-equivalences in $\cD$.
\end{defn}

\begin{rmk}
\label{InvertibilitySetIsBiInvertibilitySet}
If $E$ is an invertibility set in an $\omega$-category $\cD$, then $E$ is also a bi-invertibility set in $\cD$.
\end{rmk}

The following is often taken as the defining property for the set $\bieq\cD$ of $\omega$-bi-equivalences in an $\omega$-category $\cD$ (cf.~in \cite[Definition~4]{RiceCoinductive}).

\begin{prop}
\label{CoinductiveDefinitionBiequivalences}
Let $\cD$ be an $\omega$-category and $n>0$. Given $a\in\cD_n$, we have that $a\in\bieq_n\cD$ if and only if there exist $ a^L, a^R\in\cD_n$ of the form
\[
 a^L, a^R:d^+_{n-1}a\to d^-_{n-1}a
\]
and $c,c'\in\bieq_{n+1}\cD$ of the form
\[c \colon   a^L\underset{n-1}*a\to \id_{d^-_{n-1} a} \quad\text{ and }\quad
c' \colon   a\underset{n-1}* a^R \to \id_{d^+_{n-1} a}.\]
\end{prop}

\begin{proof}
For the forward direction, we suppose that $a\in\bieq\cD$. By \cref{BiEquivalence} there exists a bi-invertibility set $E$ containing $a$, and by \cref{BiInvertibilitySet} there exist $ a^L, a^R\in \cD_n$,
and $c,c'\in E_{n+1}$ of the form displayed in \cref{BiInvertibilitySet}. Since $c,c'\in E$, by \cref{BiEquivalence} it follows that $c,c'\in\bieq_n\cD$, as desired.

For the converse direction, suppose that for a given $a\in\cD_n$ there exist $ a^L, a^R\in \cD_n$, $c,c'\in\bieq_{n+1}\cD$ satisfying the conditions of the statement. By \cref{BiEquivalence} there exist bi-invertibility sets $E$ and $E'$ in $\cD$ containing $c$ and $c'$, respectively. Then $E''\coloneqq\{a\}\cup E\cup E'$ is by \cref{BiInvertibilitySet} an invertibility set containing $a$. By \cref{BiEquivalence}, it follows that $a\in\bieq_n\cD$, as desired.
\end{proof}

We now establish some closure properties of the set of biequivalences in an $\omega$-category $\cD$, which are essentially the content of \cite[Theorem 13]{RiceCoinductive}.

\begin{lem}
\label{BiInvertibilitySetOfIdentities}
Let $\cD$ be an $\omega$-category. If we denote
\[
\id_n\cD\coloneqq\{\id_{a}\in\cD_{n}\ |\ a\in\cD_{n-k},\ k>0\},
\]
the set $\id\cD\coloneqq\coprod_{n>0}\id\cD$ is a bi-invertibility set.
\end{lem}

\begin{proof}
This is straightforward from \cref{BiInvertibilitySet}.
\end{proof}

\begin{prop}
\label{IdentityIsBiequivalence}
Let $\cD$ be an $\omega$-category and $n\geq0$. Given $a\in\cD_n$, we have that $\id_a\in\bieq_{n+1}\cD$.
\end{prop}

\begin{proof}
A bi-invertibility set in the sense of \cref{BiInvertibilitySet} containing $\id_a$ is constructed in \cref{BiInvertibilitySetOfIdentities}. It follows from \cref{BiEquivalence} that $\id_a\in\bieq_{n+1}\cD$, as desired.
\end{proof}

\begin{lem}
\label{InvertibilitySetOfHorizontalComposites}
Let $\cD$ be an $\omega$-category. If we denote
\[
E_n\coloneqq\{b*_{k}a\ |\ a,b\in\bieq_n\cD,\ 0\leq k<n-1\},
\]
the set $E\coloneqq\coprod_{n>0}E_n$ is a bi-invertibility set.
\end{lem}

\begin{proof}
 Given $e\coloneqq b*_ka\in E_n$, by \cref{CoinductiveDefinitionBiequivalences} there exist $ a^L, a^R, b^L, b^R\in\cD_{n}$, $c,c',d,d'\in \bieq_{n+1}\cD$
of the form 
\[c \colon    a^L\underset{n-1}*a \to \id_{d^-_{n-1} a}\quad\text{ and }\quad c' \colon  a\underset{n-1}* a^R\to \id_{d^+_{n-1} a},\]
\[
d \colon   b^L\underset{n-1}*b\to \id_{d^-_{n-1} b}  \quad\text{ and }\quad d' \colon b\underset{n-1}* b^R \to \id_{d^+_{n-1} b} .
\]
We then define $ e^R\coloneqq  b^R*_k  a^R\in \cD_{n}$ and $ e^L\coloneqq  b^L*_k  a^L\in \cD_{n}$,
and we set $\ell\in \cD_{n+1}$ and $\ell'\in \cD_{n+1}$ to be the composites 
\[\quad \ell\coloneqq d\underset{k}* c\colon    e^L\underset{n-1}*e\to \id_{d^-_{n-1} e}\quad\text{ and }\quad \ell'\coloneqq d'\underset{k}* c'\colon  e\underset{n-1}* e^R \to \id_{d^+_{n-1} e}.\]
These composites do make sense because various relations, such as an instance of the interchange law
\[e*_{n-1}e^R=(b*_k a)*_{n-1}(b^R *_k a^R)=(b*_{n-1}b^R)*_k (a*_{n-1}a^R)\]
hold.
By definition
we see that $\ell\in E_{n+1}$ and $\ell'\in E_{n+1}$, so it follows that $E$ is a bi-invertibility set containing $e$, as desired.
\end{proof}

\begin{prop}
\label{HorizontalCompositeOfBiequivalenceIsBiequivalence}
Let $\cD$ be an $\omega$-category and $0\leq k<n-1$. 
Given $a,b\in\bieq_n\cD$ such that $b*_{k}a$ is defined,
we have that $b*_{k}a\in\bieq_n\cD$.
\end{prop}

\begin{proof}
A bi-invertibility set in the sense of \cref{BiInvertibilitySet} containing $b*_ka$ is constructed in \cref{InvertibilitySetOfHorizontalComposites}. It follows from \cref{BiEquivalence} that $b*_ka\in\bieq_n\cD$, as desired.
\end{proof}

\begin{lem}
\label{InvertibilitySetOfVerticalComposites}
Let $\cD$ be an $\omega$-category. If we denote
\[
E_n\coloneqq\{b\underset{n-1}*a \in \cD_n\ |\ a,b\in\bieq_n\cD\},
\]
the set $E\coloneqq\coprod_{n>0}E_n$ is a bi-invertibility set.
\end{lem}

\begin{proof}
Given $e\coloneqq b*_{n-1}a\in E_n$, by \cref{CoinductiveDefinitionBiequivalences} there exist $ a^L, a^R, b^L, b^R\in\cD_n$ and $c,c',d,d'\in\bieq_{n+1}\cD$
of the form
\[
c \colon     a^L\underset{n-1}*a \to \id_{d^-_{n-1} a} \quad\text{ and }\quad c' \colon   a\underset{n-1}* a^R  \to \id_{d^+_{n-1} a};\]
\[
d \colon   b^L\underset{n-1}*b\to \id_{d^-_{n-1} b}  
\quad\text{ and }\quad
d' \colon  b\underset{n-1}* b^R \to \id_{d^+_{n-1} b} .
\]
We then define $ e^L\coloneqq a^L*_{n-1} b^L\in\cD_n$ and $ e^R\coloneqq a^R*_{n-1} b^R\in\cD_n$, and 
set $\ell\in\cD_{n+1}$ and $\ell'\in\cD_{n+1}$ to be the composites
\[\ell\coloneqq  c\underset{n}*(\id_{ a^L}\underset{n-1}* d \underset{n-1}* \id_{a}) \colon e^L\underset{n-1}*e \to \id_{d^-_{n-1} e} 
\]
\[
\ell'\coloneqq  d'\underset{n}*(\id_{b}\underset{n-1}* c' \underset{n-1}* \id_{b^R}) \colon  e\underset{n-1}* e^R \to \id_{d^+_{n-1} e}.\]
These composites do make sense because composition is associative and various relations, such as $d_{n-1}^-a=d_{n-1}^-e$, hold.
By \cref{IdentityIsBiequivalence,HorizontalCompositeOfBiequivalenceIsBiequivalence} we can recognize that $\ell$ and $\ell'$ are composites of $\omega$-bi-equivalences of dimension $n+1$ along cells of dimension $n$, so by definition of $E$
we obtain that $\ell,\ell'\in E_{n+1}$.
So $E$ is a bi-invertibility
set containing $e$, as desired.
\end{proof}

\begin{prop}
\label{VerticalCompositeOfBiequivalenceIsBiequivalence}
Let $\cD$ be an $\omega$-category and $n>0$. Given $a,b\in\bieq_n\cD$ such that $b*_{n-1}a\in\cD_n$, we have that $b*_{n-1}a\in\bieq_n\cD$.
\end{prop}

\begin{proof}
A bi-invertibility set in the sense of \cref{BiInvertibilitySet} containing $b*_{n-1}a$ is constructed in \cref{InvertibilitySetOfVerticalComposites}. It follows from \cref{BiEquivalence} that $b*_{n-1}a\in\bieq_n\cD$, as desired.
\end{proof}

\begin{lem}
\label{BiInvertibilitySetOfHalfInverses}
Let $\cD$ be an $\omega$-category. If we denote
\[
E_n\coloneqq\{ b\underset{n-1}*a^L\in\cD_n\ |\ a,b\in\bieq_n\cD,\  a^L\text{ is a left inverse for }a\},
\]
then the set $E\coloneqq\coprod_{n>0}E_n$ is a bi-invertibility set.
\end{lem}

\begin{proof}
Given $e\coloneqq b\underset{n-1}*a^L\in E_n$ for $a,b\in\bieq_n\cD$, by \cref{CoinductiveDefinitionBiequivalences} there exist $a^R, b^L, b^R\in\cD_n$ and $c,c',d,d'\in\bieq_{n+1}\cD$
of the form
\[
c \colon   a^L\underset{n-1}*a \to  \id_{d^-_{n-1} a}\quad\text{ and }\quad c' \colon   a\underset{n-1}* a^R \to \id_{d^+_{n-1} a},\]
\[
d \colon     b^L\underset{n-1}*b \to \id_{d^-_{n-1} b}\quad\text{ and }\quad d' \colon   b\underset{n-1}* b^R\to \id_{d^+_{n-1} b}.
\]

We first consider $x,y\in\cD_{n+1}$ defined as follows: 
\[
\begin{tikzcd}
y	\colon {a \nunderset{n-1}*
b^L\nunderset{n-1}*b\nunderset{n-1}* a^L \nunderset{n-1}* a \nunderset{n-1}* a^R }
\arrow[r,"{{\id}_{a}\underset{n-1}* d \underset{n-1}* c \underset{n-1}*\id_{a^R}}"]
&[2.5cm] {a\underset{n-1}*a^R} \arrow[r, "{c'}"]& {\id_{d^+_{n-1} a}} 
	\end{tikzcd}\]
\[x\colon a \nunderset{n-1}*
b^L\nunderset{n-1}*b\nunderset{n-1}* a^L \nunderset{n-1}* a \nunderset{n-1}* a^R 
\xrightarrow{{\id}_{a \underset{n-1}*b^L\underset{n-1}*b\underset{n-1}* a^L }\underset{n-1}*c'} a \underset{n-1}*b^L\underset{n-1}*b\underset{n-1}* a^L
\]
By \cref{VerticalCompositeOfBiequivalenceIsBiequivalence,HorizontalCompositeOfBiequivalenceIsBiequivalence,IdentityIsBiequivalence},
we know that $x,y\in\bieq_{n+1}\cD$.
If $x^L$ denotes a left inverse for $x$, we then define $ e^L\coloneqq  a*_{n-1}b^L\in\cD_n$ and $ e^R\coloneqq  a*_{n-1} b^R\in\cD_n$, and
set $\ell\in\cD_{n+1}$ and $\ell'\in\cD_{n+1}$ to be the composites
\[\ell\colon  e^L\underset{n-1}*e \xrightarrow{ x^L}
a\underset{n-1}*b^L\underset{n-1}*b\underset{n-1}* a^L\underset{n-1}*a\underset{n-1}* a^R\xrightarrow{y}\id_{d^-_{n-1} e}
,
\]

\[	\ell'\colon e\underset{n-1}* e^R 
\xrightarrow{\id_{b}\underset{n-1}*c\underset{n-1}*\id_{b^R}} 
b\underset{n-1}*b^R
\xrightarrow{d'}
\id_{d^+_{n-1} e}
\]
By construction, we see that $\ell\in E_{n+1}$.
By \cref{HorizontalCompositeOfBiequivalenceIsBiequivalence,IdentityIsBiequivalence,VerticalCompositeOfBiequivalenceIsBiequivalence}, we see that $\ell'\in\bieq_{n+1}\cD$, and in particular $\ell'=\ell'
*_n \id_{d_n^-\ell'}
\in E_{n+1}$,
so we get that $E$ is a bi-invertibility set containing $e$, as desired.
\end{proof}

\begin{prop}[{\cite[Lemma 14]{RiceCoinductive}}]
\label{LeftAndRightInverseIsBiequivalence}
Let $\cD$ be an $\omega$-category and $n>0$. Given $a\in\bieq_n\cD$, if $ a^L$ and $ a^R$ are, respectively, a left and right weak inverse for $a$, then $ a^L, a^R\in\bieq_n\cD$.
\end{prop}

\begin{proof}
A bi-invertibility set in the sense of \cref{BiInvertibilitySet} containing $a^L$ is constructed in \cref{BiInvertibilitySetOfHalfInverses}, and one for $a^R$ can be constructed with a similar argument. It follows from \cref{BiEquivalence} that $ a^L, a^R\in\bieq_n\cD$, as desired.
\end{proof}

\begin{lem}
\label{InvertibilitySetOfBiequivalences}
Given an $\omega$-category $\cD$,
we have that $\bieq\cD\coloneqq\coprod_{n>0}\bieq_n\cD$ is an invertibility set.
\end{lem}

\begin{proof}
Given $a\in \bieq_n\cD$, by \cref{BiInvertibilitySet} there exist $ a^L, a^R\in\cD_{n}$ and $c,c'\in \bieq_{n+1}\cD$ of the form
\[c \colon   a^L\underset{n-1}*a \to \id_{d^-_{n-1} a} \quad\text{ and }\quad
c'\colon a\underset{n-1}* a^R\to \id_{d^+_{n-1} a} .\]
If $ {c'}^L\in\cD_{n+1}$ is a left inverse for $c'$, we set $\ell\in\cD_{n+1}$ to be the composite
\[
\ell\colon {a\underset{n-1}* a^L} \xrightarrow{\id_{a\underset{n-1}* a^L}\underset{n-1}*{c'}^L} {a\nunderset{n-1}* a^L\nunderset{n-1}* a\nunderset{n-1}* a^R}  \xrightarrow{\id_a\underset{n-1}*c\underset{n-1}*\id_{ a^R}} {a\underset{n-1}* a^R}  \xrightarrow{c'} {\id_{d_{n-1}^+a}}.
\]
By \cref{LeftAndRightInverseIsBiequivalence} we know that $ a^L\in\bieq_{n}\cD$, and
by \cref{IdentityIsBiequivalence,VerticalCompositeOfBiequivalenceIsBiequivalence,HorizontalCompositeOfBiequivalenceIsBiequivalence,LeftAndRightInverseIsBiequivalence}
we know that $\ell\in \bieq_{n+1}\cD$.
Given that we also have that $c\in\bieq_{n+1}\cD$, this shows that $\bieq\cD$ is an invertibility set, as desired.
\end{proof}

\begin{prop}[{\cite[Corollary 19]{RiceCoinductive}}]
\label{prop:invertible is the same as biinvertible}
Let $\cD$ be an $\omega$-category and $n>0$. Given $a\in\cD_n$, we have that $a\in\eq_n\cD$ if and only if $a\in\bieq_n\cD$.
\end{prop}

\begin{proof}
If $a\in\eq_n\cD$ (resp.~$a\in\bieq_n\cD$), a bi-invertibility set (resp.~invertibility set) containing $a$ is constructed in \cref{InvertibilitySetIsBiInvertibilitySet,InvertibilitySetOfEquivalences}  (resp.~\cref{InvertibilitySetOfBiequivalences}). It follows from \cref{BiEquivalence} (resp.~\cref{Equivalence}) that $a\in\bieq_n\cD$ (resp.~$a\in\eq_n\cD$), as desired.
\end{proof}

\subsection{The homotopy theory of \omegat-categories}

\begin{thm}[{\cite[\textsection4,5]{LMW}}]
There exists a model structure on the category $\omega\cat$ of $\omega$-categories, which we denote $\omega\cat_{\mathrm{can}}$ and call the \emph{canonical model structure}, in which:
\begin{itemize}[leftmargin=*]
\item every object is fibrant.
\item the class of cofibrations is generated by the set of boundary inclusions $\partial\cC_n\hookrightarrow\cC_n$ for $n\geq0$.
\item the cofibrant objects are precisely the polygraphs, considered e.g.~in \cite[\textsection5]{LMW}.
\end{itemize}
\end{thm}

\begin{proof}
The model structure $\omega\cat_{\mathrm{can}}$ is constructed in \cite[Theorem~4.39]{LMW}, and the description of the fibrant and cofibrant objects can be found in  \cite[\textsection 5]{LMW}.
\end{proof}

\subsection{The model for the coherent \omegat-equivalence}

\begin{const}
\label{constructionJ}
We denote by $\cQ$ the free category generated by three $1$-cells $f\colon p\to q$, $g\colon  q\to p$ and $g'\colon q\to p$.
 This is obtained by gluing $f$ ``head-to-tail'' with both $g$ and $g'$,
and generating all possible compositions.
The set of objects is $\Ob\cQ=\{p, q\}$.
The category $\cQ$ as a whole can be understood as the pushout in $\omega\cat$
\[
\begin{tikzcd}[column sep=3.15cm]
\partial\cC_1^{\circ}\amalg\partial\cC_1\amalg\partial\cC_1^{\circ}\arrow[r,""]\arrow[d,hook] 
&\cC_0\amalg\cC_0\arrow[d,hook]\\
\cC_1^{\circ}\amalg\cC_1\amalg\cC_1^{\circ}\arrow[r]&\cQ
\end{tikzcd}
\]
\end{const}

\begin{const}
\label{Jtruncated}
Let $\widehat{\omega\cE}^{(0)} \coloneqq \cC_0 \amalg \cC_0$.
For $\indind > 0$, 
we define inductively $\widehat{\omega\cE}^{(\indind)}$ to be an $\omega$-category (in fact a $\indind$-category) coming with a triple of $\omega$-functors
\[\imath_{\indind}\colon \widehat{\omega\cE}^{(\indind-1)} \to \widehat{\omega\cE}^{(\indind)}\quad\text{ and }\quad\firstiso_{\indind}, \secondiso_{\indind}\colon \Sigma(\widehat{\omega\cE}^{(\indind-1)}) \to \widehat{\omega\cE}^{(\indind)}.\]

\begin{itemize}[leftmargin=*]
\item For $\indind = 1$, we let $\widehat{\omega\cE}^{(1)} \coloneqq \cQ$, we let $\imath_1$ be the inclusion
\[
\widehat{\omega\cE}^{(0)}=\cC_0\amalg\cC_0 \hookrightarrow\cQ= \widehat{\omega\cE}^{(1)} 
\]
and we let $\firstiso_1$ and $\secondiso_1$
\[
\firstiso_1\colon\Sigma\widehat{\omega\cE}^{(0)}=\cC_1\amalg\cC_1 \to\cQ= \widehat{\omega\cE}^{(1)}\quad\text{ and }\quad
\secondiso_1\colon
\Sigma\widehat{\omega\cE}^{(0)}=\cC_1\amalg\cC_1\to\cQ= \widehat{\omega\cE}^{(1)} 
\]
be the $\omega$-functors determined by
\[
    \firstiso_1\colon \Sigma p \mapsto g \underset{0}* f, \; \Sigma q \mapsto \id_p,\quad\text{ and }\quad
    \secondiso_1\colon \Sigma p \mapsto f \underset{0}* g', \; \Sigma q \mapsto \id_q.
\]

\item For $\indind > 1$, we let $\widehat{\omega\cE}^{(\indind)}$, $\imath_\indind$, $\firstiso_\indind$, and $\secondiso_\indind$ be defined by the pushout in $\omega\cat$
\begin{equation}
    \label{omegacen_pushout}
\begin{tikzcd}
	{\Sigma(\widehat{\omega\cE}^{(\indind-2)}) \amalg \Sigma(\widehat{\omega\cE}^{(\indind-2)})} &&& {\widehat{\omega\cE}^{(\indind-1)}} \\
	{\Sigma(\widehat{\omega\cE}^{(\indind-1)}) \amalg \Sigma(\widehat{\omega\cE}^{(\indind-1)})} &&& {\widehat{\omega\cE}^{(\indind)}}.
	\arrow["{[\firstiso_{\indind-1}, \secondiso_{\indind-1}]}", from=1-1, to=1-4]
	\arrow["{\imath_{\indind}}", from=1-4, to=2-4]
	\arrow["{\Sigma(\imath_{\indind-1}) \amalg \Sigma (\imath_{\indind-1})}", from=1-1, to=2-1,swap]
	\arrow["{[\firstiso_{\indind}, \secondiso_{\indind}]}", from=2-1, to=2-4,swap]
\end{tikzcd}
\end{equation}
\end{itemize}
\end{const}

\begin{const}
\label{constomegahat}
We denote by $\widehat{\omega\cE}$ the $\omega$-category
obtained as the colimit in $\omega\cat$
\[\widehat{\omega\cE}\coloneqq \colim[\quad\dots\hookleftarrow\widehat{\omega\cE}^{(\indind+1)}\xhookleftarrow{\imath_{\indind+1}} \widehat{\omega\cE}^{(\indind)}\hookleftarrow\dots\hookleftarrow\widehat{\omega\cE}^{(2)}\xhookleftarrow{\imath_2}\widehat{\omega\cE}^{(1)}\xhookleftarrow{\imath_1}\widehat{\omega\cE}^{(0)}\quad].\]

\begin{rmk}
\label{alphainfinity}
The $\omega$-functors $\firstiso_k,\secondiso_k\colon\Sigma(\widehat{\omega\cE}^{(k-1)})\to\widehat{\omega\cE}^{(k)}$
induce $\omega$-functors
\[
\firstiso_{\infty},\secondiso_{\infty}\colon\Sigma(\widehat{\omega\cE})\to\widehat{\omega\cE}.\]
\end{rmk}
\end{const}

The following result justifies the name of walking $\omega$-equivalence.

\begin{prop}
\label{walkingBiinvertibility}
Let $\cD$ be an $\omega$-category. Given $a \in \cD_n$, we have that $a \in \bieq_n\cD$ if and only if  there exists an $\omega$-functor $\tilde a\colon \Sigma^{n-1}(\widehat{\omega\cE}) \to \cD$ such that the following diagram commutes:
\begin{equation} \label{walking_biequivalence_lifting}
\begin{tikzcd}
	{\cC_n} && \cD \\
	{\Sigma^{n-1}(\widehat{\omega\cE})}
	\arrow["{\Sigma^{n-1}f}"', from=1-1, to=2-1]
	\arrow["a", from=1-1, to=1-3]
	\arrow["{\tilde{a}}"', from=2-1, to=1-3]
\end{tikzcd}
\end{equation}
\end{prop}
\begin{proof}
For each $n\geq0$ and $a \in \bieq_n\cD$, make a choice of $a^L, a^R\in\cD_{n}$ and of $c_a,c'_a \in \bieq_{n+1}\cD$
of the form
\[
    c_a \colon  a^L\underset{n-1}*a \to \id_{d^-_{n-1} a} \quad\text{ and }\quad c'_a \colon a\underset{n-1}* a^R \to \id_{d^+_{n-1} a}.
\]
By recursion on $k \geq 0$, we construct families of $\omega$-functors
\[
    \tilde{a}^{(k)}\colon \Sigma^{n-1}\widehat{\omega\cE}^{(k)} \to \cD
\]
parameterized by $n \geq0$ and $a \in \bieq_n\cD$, such that
\[\begin{tikzcd}
	{\cC_n} && \cD \\
	{\Sigma^{n-1}(\widehat{\omega\cE}^{(k)})}
	\arrow["{\Sigma^{n-1}f}"', from=1-1, to=2-1]
	\arrow["a", from=1-1, to=1-3]
	\arrow["{\tilde{a}^{(k)}}"', from=2-1, to=1-3]
\end{tikzcd}\]
commutes, and satisfying 
\begin{equation}
 \label{walking_bieq_compatibility_1e2}  
    \tilde{a}^{(k-1)}  = \tilde{a}^{(k)} \circ \Sigma^{n-1}(\imath_{k})\text{ and }
    [\tilde{c}_a^{(k-1)}, \tilde{c'}_a^{(k-1)}]  = \tilde{a}^{(k)} \circ [\Sigma^{n-1}(\firstiso_{k}), \Sigma^{n-1}(\secondiso_{k})]
\end{equation}
for all $k > 0$.
For each $n \in \mathbb{N}$ and $a \in \bieq_n\cD$, we let $\tilde{a}^{(1)}$ be defined by
\[
    \Sigma^{n-1}f \mapsto a, \quad \Sigma^{n-1}g \mapsto a^L, \quad \Sigma^{n-1}g' \mapsto a^R,
\]
and set $\tilde{a}^{(0)} \coloneqq \tilde{a}^{(1)} \circ \Sigma^{n-1}(\imath_{1})$.
Then the equality
\[[\tilde{c}_a^{(0)}, \tilde{c'}_a^{(0)}] = \tilde{a}^{(1)} \circ [\Sigma^{n-1}(\firstiso_{1}), \Sigma^{n-1}(\secondiso_{1})]\]
holds by construction.

Let $k > 1$, $n \in \mathbb{N}$, and $a \in \bieq_n\cD$.
By the inductive hypothesis, we have a commutative diagram in $\omega\cat$
\[
\begin{tikzcd}
	{\Sigma^n(\widehat{\omega\cE}^{(k-2)}) \amalg \Sigma^n(\widehat{\omega\cE}^{(k-2)})} &&&& {\Sigma^{n-1}(\widehat{\omega\cE}^{(k-1)})} \\
	{\Sigma^n(\widehat{\omega\cE}^{(k-1)}) \amalg \Sigma^n(\widehat{\omega\cE}^{(k-1)})} &&&& {\cD}.
	\arrow["{[\Sigma^{n-1}(\firstiso_{k-1}), \Sigma^{n-1}(\secondiso_{k-1})]}", from=1-1, to=1-5]
	\arrow["{\tilde{a}^{(k-1)}}", from=1-5, to=2-5]
	\arrow["{\Sigma^n(\imath_{k-1}) \amalg \Sigma^n(\imath_{k-1})}", from=1-1, to=2-1]
	\arrow["{[\tilde{c}_a^{(k-1)}, \tilde{c'}_a^{(k-1)}]}", from=2-1, to=2-5]
\end{tikzcd}
\]
Using the universal property of the pushout (\ref{omegacen_pushout}) and the fact that $\Sigma^{n-1}$ preserves pushouts by \cref{AdjointOfSuspension},
we see that this diagram induces a unique $\omega$-functor
\[\tilde{a}^{(k)}\colon \Sigma^{n-1}\widehat{\omega\cE}^{(k)} \to \cD\] satisfying \eqref{walking_bieq_compatibility_1e2}.
This completes the inductive step. 
Since $\Sigma^{n-1}$ preserves sequential colimits by \cref{AdjointOfSuspension}, for each $a \in \bieq_n \cD$, we obtain universally an $\omega$-functor $\tilde{a}\colon \Sigma^{n-1}(\widehat{\omega\cE}) \to \cD$ such that (\ref{walking_biequivalence_lifting}) commutes.

Conversely, for each $n\geq0$, let
    \[
        E_n \coloneqq \{ a \in \cD_n \mid \text{there exists $\tilde a\colon \Sigma^{n-1}\widehat{\omega\cE} \to \cD$ such that $a = \tilde{a} \circ \Sigma^{n-1} f$} \};
    \]
    we will show that $E \coloneqq \coprod_{n \geq 0} E_n$ is a bi-invertibility set.
    Let $a \in E_n$.
    By definition there exists $\tilde{a}$ such that $a = \tilde{a} \circ \Sigma^{n-1} f$.
    In particular, there are $(n+1)$-cells
    \[
        c \colon  a^L\underset{n-1}*a \to \id_{d^-_{n-1} a} \quad\text{ and }\quad c' \colon a\underset{n-1}* a^R \to \id_{d^+_{n-1} a}
    \]
    in the image of $\Sigma^{n-1}\widehat{\omega\cE}^{(2)}$ through $\tilde{a}$.
    Then
    \[
        \tilde{c} \coloneqq \tilde{a} \circ \Sigma^{n-1}(\firstiso_\infty)\colon \Sigma^n\widehat{\omega\cE} \to \cD \quad \text{and} \quad
        \tilde{c}' \coloneqq \tilde{a} \circ \Sigma^{n-1}(\secondiso_\infty)\colon \Sigma^n\widehat{\omega\cE} \to \cD
    \]
    are $\omega$-functors satisfying
    \[
        c = \tilde{c} \circ \Sigma^n f, \quad \quad c' = \tilde{c}' \circ \Sigma^n f.
    \]
    It follows that $c, c' \in E_{n+1}$.
    This completes the proof.
\end{proof}

\begin{rmk} 
\label{rmk:omegahat_is_polygraph}
By construction, $\widehat{\omega\cE}$ is a polygraph, whose set of $k$-cells is freely generated by the set $E_k$ defined, inductively on $k$, by
\begin{equation}
\label{eq:omegahat_is_polygraph}
    E_0 \coloneqq \{p, q\}, \; E_1 \coloneqq \{f, g, g'\}, \;
    E_k \coloneqq \firstiso_\infty (\Sigma E_{k-1}) \cup \secondiso_\infty (\Sigma E_{k-1})
    \text{ for $k > 1$.}
\end{equation}
\end{rmk}

\begin{lem} \label{lem:all_generators_of_omegahat_are_bieq}
    With reference to the notation of \eqref{eq:omegahat_is_polygraph}, let $n > 0$ and $a \in E_n$.
    Then $a\in\bieq_n\widehat{\omega\cE}$.
\end{lem}
\begin{proof}
    First, suppose that $n = 1$ and $a = f$.
    Then the classifying $\omega$-functor $f\colon \cC_1 \to \widehat{\omega\cE}$ factors as $\id_{\widehat{\omega\cE}} \circ f$
    as in 
    \[
    \begin{tikzcd}
	{\cC_1} && \widehat{\omega\cE} \\
	{\Sigma^{0}(\widehat{\omega\cE})}
	\arrow["{\Sigma^{0}f}"', from=1-1, to=2-1]
	\arrow["f", from=1-1, to=1-3]
	\arrow["{\id}"', from=2-1, to=1-3]
\end{tikzcd}
    \]
    So, by \cref{walkingBiinvertibility} $f\in\bieq_1\widehat{\omega\cE}$.
    If $a = g$ or $a = g'$, then $a$ is a left or right weak inverse of $f$, so by Proposition \ref{LeftAndRightInverseIsBiequivalence}, the $1$-morphism $a$ is also a biequivalence.
    
    Now, suppose that $n > 1$.
    Then there exists $e\in E_{n-1}$ such that $a = \firstiso_\infty(\Sigma e)$ or $a = \secondiso_\infty(\Sigma e)$, and by the inductive hypothesis $e\in\bieq_{n-1}\widehat{\omega\cE}$.
    By Proposition \ref{walkingBiinvertibility}, there exists $\tilde{e}\colon \Sigma^{n-2}\widehat{\omega\cE} \to \widehat{\omega\cE}$ such that $e$ factors
    as in
        \[
    \begin{tikzcd}
	{\cC_{n-1}} && \widehat{\omega\cE} \\
	{\Sigma^{n-2}(\widehat{\omega\cE})}
	\arrow["{\Sigma^{n-2}f}"', from=1-1, to=2-1]
	\arrow["e", from=1-1, to=1-3]
	\arrow["{\tilde{e}}"', from=2-1, to=1-3]
\end{tikzcd}
    \]
    Assume without loss of generality that $a = \firstiso_\infty(\Sigma e)$.
    Then, letting $\tilde{a} \coloneqq \firstiso_\infty \circ \Sigma \tilde{e}$, we have that
    \[
        a = \firstiso_\infty \circ \Sigma (\tilde{e} \circ \Sigma^{n-2}f) = (\firstiso_\infty \circ \Sigma \tilde{e}) \circ \Sigma(\Sigma^{n-2}f) = \tilde{a} \circ \Sigma^{n-1}f.
    \]
 so $a$ factors as in
    \[
        \begin{tikzcd}
	{\cC_{n}} && \widehat{\omega\cE} \\
	{\Sigma^{n-1}(\widehat{\omega\cE})}
	\arrow["{\Sigma^{n-1}f}"', from=1-1, to=2-1]
	\arrow["a", from=1-1, to=1-3]
	\arrow["{\tilde{a}}"', from=2-1, to=1-3]
\end{tikzcd}
    \]
 Hence, we conclude by \cref{walkingBiinvertibility} that $a\in\bieq_n\widehat{\omega\cE}$, as desired.
\end{proof}

\begin{prop} \label{prop:all_cells_of_omegahat_are_equivalences}
    Let $n > 0$ and $a \in (\widehat{\omega\cE})_n$.
    Then $a\in\bieq_n\widehat{\omega\cE}$.
\end{prop}
\begin{proof}
    By \cref{rmk:omegahat_is_polygraph}, the cells of $\widehat{\omega\cE}$ are composition-generated, in the sense of \cite[Proposition 15.1.8]{PolyBook},
    by the cells in $E \coloneqq \coprod_{k \geq 0} E_k$.
    By Lemma \ref{lem:all_generators_of_omegahat_are_bieq}, all the generators are biequivalences, and by Proposition \ref{VerticalCompositeOfBiequivalenceIsBiequivalence} and Proposition \ref{HorizontalCompositeOfBiequivalenceIsBiequivalence}, biequivalences are closed under composition.
\end{proof}

The remainder of the paper is devoted to proving the following.

\begin{thm}
\label{omegaEcontractible}
    The unique $\omega$-functor
    $\widehat{\omega\cE}\to\cC_0$
    is a weak equivalence in $\omega\cat_{\mathrm{can}}$.
\end{thm}

\section{The marked model for the coherent \omegat-equivalence}

\subsection{Marked \omegat-categories}

We briefly recall some notions on marked $\omega$-categories from \cite[\textsection2]{HenryLoubaton} that will be needed in this paper.

A \emph{marked $\omega$-category} is a pair $(\cD,t\cD)$ where $\cD$ is an $\omega$-category and $t\cD\coloneqq \coprod_{n>0}t\cD_{n}$ is a sequence of sets such that for any $n>0$, the set $t\cD_n$ is a subset of $\cD_n$ containing identities and closed under composition.  The $\omega$-category $\cD$ is called \emph{the underlying $\omega$-category} and $t\cD$ \emph{the marking} of $\cD$. A cell in $t\cD$ is called \emph{marked}. A \emph{marked $\omega$-functor} $F\colon(\cD,t\cD)\to(\cE,t\cE)$ consists of a marking-preserving $\omega$-functor $F\colon\cD\to\cE$. We denote $\omegacat^+$ the category of marked $\omega$-categories and marked $\omega$-functors. The assignment $(\cD,t\cD)\mapsto\cD$ of the underlying $\omega$-category of any marked $\omega$-category defines a forgetful functor $U\colon\omegacat^+\to\omega\cat$.

\begin{notn}
Given an $\omega$-category $\cD$, one can consider various choices of interest for the marking on $\cD$:
\begin{itemize}[leftmargin=*, label={>>}]
    \item If $\id\cD$
    denotes the set of identities of $\cD$, the class $\id\cD$ is closed under composition and contains identities. So, $(\cD,\id\cD)\eqqcolon\cD^{\flat}$ is a marked $\omega$-category. The assignment $\cD\mapsto\cD^{\flat}$ defines a functor $(-)^{\flat}\colon\omega\cat\to\omegacat^+$. 
    \item If $\mor\cD$ denotes the set of cells of $\cD$ of strictly positive dimension, the class $\mor\cD$ is closed under composition and contains identities. So, $(\cD,\mor\cD)\eqqcolon\cD^{\sharp}$ is a marked $\omega$-category. The assignment $\cD\mapsto\cD^{\sharp}$ defines a functor $(-)^{\sharp}\colon\omega\cat\to\omegacat^+$.
     \item If $\operatorname{eq}\cD$ denotes the set of $\omega$-equivalences of $\cD$ as in \cref{Equivalence},
     by \cite[Lemma 20.1.4]{PolyBook} the class $\eq\cD$ is closed under composition and contains identities. So, $(\cD,\eq\cD)\eqqcolon\cD^{\natural}$ is a marked $\omega$-category. By \cref{walkingBiinvertibility}, the assignment $\cD\mapsto\cD^{\natural}$ defines a functor $(-)^{\natural}\colon\omega\cat\to\omegacat^+$.
\end{itemize}
\end{notn}

The following adjoint pairs can be checked by verifying the appropriate universal properties, and using \cref{AdjointOfSuspension} for the second one.

\begin{prop}
\label{AdjointsOfU}
There are adjunctions
\[(-)^{\flat}\colon\omega\cat\rightleftarrows\omegacat^+ :\! U\quad\text{ and }\quad U\colon\omegacat^+\rightleftarrows\omega\cat :\!(-)^{\sharp}.\]
In particular, the functor $U\colon\omegacat^+\to\omega\cat$ preserves limits and colimits.
\end{prop}

\begin{prop}
\label{MarkedAdjointOfSuspension}
If $\omegacat^+_{*,*}$ denotes the category of bipointed marked $\infty$-categories,
there is an adjunction
\[\Sigma\colon\omegacat^+\rightleftarrows\omegacat^+_{*,*} :\!\hom\]
Moreover,
the functor $\Sigma\colon\omegacat^+\to\omegacat^+$ preserves connected colimits.
\end{prop}

\subsection{The coinductive homotopy theory of marked \omegat-categories}

We recall that a \emph{left semi-model category structure} on a category $\cM$ consists of three distinguished classes of morphisms of $\cM$, called \emph{cofibrations}, \emph{fibrations}, and \emph{weak equivalences}, satisfying a weaker version of the axioms for a model category.
We refer the reader to \cite[Definition~2.1]{BataninWhiteLoc} for a complete list of axioms that these classes must satisfy. An object in $\cM$ is said to be \emph{fibrant} if the unique morphism to the terminal object of $\cM$ is a fibration, and it is said to be \emph{cofibrant} if the unique morphism from the initial object of $\cM$ is a cofibration. The class of \emph{acyclic cofibrations} is the class of morphisms in $\cM$ that have the left lifting property with respect to all fibrations between fibrant objects. In a left semi-model structure, the class of acyclic cofibrations is closed under transfinite composition and pushouts and the class of weak equivalences is closed under two-out-of-three.

\begin{thm}[{\cite[\textsection4.2]{HenryLoubaton}}]
\label{thm:model structure on marked omega categories}
There exists a left semi-model structure on $\omegacat^+$, which we denote by $\omegacat^+_{\mathrm{coind}}$ and we call the \emph{coinductive left semi-model structure}, such that:
\begin{enumerate}[leftmargin=*, ref=(\arabic*)]
    \item \label{MarkCofib} a marked $\omega$-functor $f\colon (\cD,t\cD)\to (\cE,t\cE)$ is a cofibration in $\omegacat^+_{\mathrm{coind}}$ if and only if 
the $\omega$-functor  $f\colon \cD\to \cE$ is a cofibration in $\omega\cat_{\mathrm{can}}$;
    \item \label{MarkCofibCofibWE}
    a cofibration $f\colon (\cD,t\cD)\to (\cE,t\cE)$ between cofibrant objects is a weak equivalence in $\omegacat^+_{\mathrm{coind}}$ if and only if it is an acyclic cofibration, that is, it has the left lifting property against fibrations between fibrants objects;
    \item\label{MarkFibObj} a marked $\omega$-category $(\cD,t\cD)$ is fibrant in $\omegacat^+_{\mathrm{coind}}$ if and only if $t\cD=\eq\cD$;
    \item \label{MarkFibWE} a marked $\omega$-functor $f\colon \cD^{\natural}\to \cE^{\natural}$
    between fibrant objects is a weak equivalence in $\omegacat^+_{\mathrm{coind}}$ if and only the $\omega$-functor $f\colon \cD\to \cE$ is a weak equivalence in $\omega\cat_{\mathrm{can}}$;
    \item \label{MarkFibFib} a marked $\omega$-functor $f\colon \cD^{\natural}\to \cE^{\natural}$ between fibrant objects is a fibration in $\omegacat^+_{\mathrm{coind}}$ if and only if it has the right lifting property against the marked $\infty$-functors of the form $i_n^+\colon\cC_n^{\flat}\to(\cC_{n+1},\{e_{n+1}\}\cup\id(\cC_{n+1}))$ for all $n\geq0$.
    Here, $e_{n+1}$ denotes the non-trivial $(n+1)$-cell of $\cC_{n+1}$ and $i_n^+$ denotes the marked $\omega$-functor that embeds $\cC_n$ as the codomain of $e_{n+1}$.
\end{enumerate}
\end{thm}

\begin{proof}
The left semi-model structure $\omegacat^+_{\mathrm{coind}}$ is built in \cite[Definition.~4.22]{HenryLoubaton} as a left Bousfield localization (in the sense of \cite[Theorem A]{BataninWhiteLoc}) of the \emph{saturated inductive left semi-model structure} from \cite[Theorem~3.31]{HenryLoubaton}. 
The saturated inductive left semi-model structure is in turn built as a left Bousfield localization of the \emph{inductive left semi-model structure} from \cite[Theorem~2.38]{HenryLoubaton}.

The characterization \ref{MarkCofib} of cofibrations  directly follows from \cite[Definition~2.27]{HenryLoubaton}.
The characterization \ref{MarkCofibCofibWE} of cofibrations between cofibrant objects that are weak equivalences  follows from \cite[Proposition 2.2.10]{HenryWeakmodelstructure}.
The characterization \ref{MarkFibObj} of fibrant objects and the characterization \ref{MarkFibWE} of weak equivalences between fibrant objects are in \cite[Theorem 4.25]{HenryLoubaton}.
The characterization \ref{MarkFibFib} of fibrations between fibrant objects then directly follows from \cite[Proposition~3.23]{HenryLoubaton}, evoking \cite[Theorem~7.3(6)]{HenryCombinatorialAccessible}
for the fact that a map between fibrant objects in the left Bousfield localization $\omega\cat_{\mathrm{coind}}$ is a fibration if and only if it is one in the inductive left semi-model structure.
\end{proof}

\begin{lem}
\label{lemma:an other important acyclic cofibration}
Given a marked $\omega$-category $(\cE,t\cE)$ with $t\cE\subseteq \eq\cE$, the canonical morphism
\[(\cE,t\cE)\hookrightarrow \cE^{\natural}\]
is an acyclic cofibration in $\omegacat^+_{\mathrm{coind}}$.
\end{lem}

\begin{proof}
In order to show that $(\cE,t\cE)\to \cE^{\natural}$ has the left lifting property with respect to any fibration between fibrant objects $p\colon\cB^{\natural}\to \cD^{\natural}$ in $\omegacat^+_{\mathrm{coind}}$, consider the following lifting problem in $\omegacat^+$:
\[
\begin{tikzcd}
(\cE,t\cE)\arrow[r]\arrow[d]&\cB^{\natural}\arrow[d]\\
\cE^{\natural}\arrow[r]\arrow[ru,dashed]&\cD^{\natural}
\end{tikzcd}\]
A lift exists (because $(-)^{\natural}\colon\omega\cat\to\omegacat^+$ is a functor), and is necessarily given by the top map at the level of underlying categories.
It follows that $(\cE,t\cE)\to \cE^{\natural}$ is an acyclic cofibration in $\omegacat^+_{\mathrm{coind}}$, as desired.
\end{proof}

\begin{notn}
Given a marked $\infty$-category $(\cD,t\cD)$, we denote by $\Sigma(\cD,t\cD)\coloneqq(\Sigma\cD,\{\Sigma a,a\in t\cD\}\cup \id(\Sigma \cD))$ the \emph{marked suspension} of $(\cD,t\cD)$.
\end{notn}

\begin{rmk}
\label{SigmaAndU}
By definition, given a marked $\infty$-category $(\cD,t\cD)$, there is a canonical isomorphism in $\omega\cat$
\[
U\Sigma(\cD,t\cD)\cong \Sigma\cD\cong \Sigma U(\cD,t\cD).
\]
\end{rmk}

\begin{prop}
\label{SuspensionHomotopical}
The functor
$\Sigma\colon\omegacat^+_{\mathrm{coind}}\to \omegacat^+_{\mathrm{coind}}$
preserves
acyclic cofibrations.
\end{prop}

\begin{proof}
We say that
\begin{itemize}[leftmargin=*]
    \item a map of $\omegacat^+_{*,*}$ is a \emph{fibration} in $(\omegacat^+_{\mathrm{coind}})_{*,*}$ if it is one in $\omegacat^+_{\mathrm{coind}}$ when ignoring the base points;
    \item an object of $\omegacat^+_{*,*}$ is \emph{fibrant} in $(\omegacat^+_{\mathrm{coind}})_{*,*}$ if it is one in $\omegacat^+_{\mathrm{coind}}$ when ignoring the base points;
    \item a map of $\omegacat^+_{*,*}$ is an \emph{acyclic cofibration} in $(\omegacat^+_{\mathrm{coind}})_{*,*}$ if it has the left lifting property with respect to all fibrations between fibrant objects.
\end{itemize}
As a preliminary observation, we argue that the functor \[U\colon(\omegacat^+_{\mathrm{coind}})_{*,*}\to\omegacat^+_{\mathrm{coind}}\]
preserves acyclic cofibrations. 
Let $j\colon (A,a,a')\to (B,b,b')$ be an acyclic cofibration in $(\omegacat^+_{\mathrm{coind}})_{*,*}$, and consider a lifting problem in $\omegacat^+_{\mathrm{coind}}$
\[
\begin{tikzcd}
A\arrow[r, "f"]\arrow[d, "j"]&X\arrow[d]\\
B\arrow[r, "g"]\arrow[ru, dashed]&Y
\end{tikzcd}
\]
This can be enhanced to a lifting problem in $(\omegacat^+_{\mathrm{coind}})_{*,*}$
\[
\begin{tikzcd}
(A,a,a')\arrow[r, "f"]\arrow[d, "j"]&(X,f(a),f(a'))\arrow[d]\\
(B,b,b')\arrow[r, "g" swap]\arrow[ru, dashed]&(X,g(b), g(b'))
\end{tikzcd}\]
This lifting problem admits a solution because, by definition, the left hand side map is an acyclic cofibration in $(\omegacat^+_{\mathrm{coind}})_{*,*}$ and the right hand side map is a fibration in $(\omegacat^+_{\mathrm{coind}})_{*,*}$.

Consider the adjunction
\[
\Sigma\colon\omegacat^+_{\mathrm{coind}}
\rightleftarrows(\omegacat^+_{\mathrm{coind}})_{*,*} :\!\hom.
\]
We first observe that the functor \[\hom\colon(\omegacat^+_{\mathrm{coind}})_{*,*}\to \omegacat^+_{\mathrm{coind}}\]
preserves fibrant objects. To see this, one can use the characterization of fibrant objects from \cref{thm:model structure on marked omega categories}\ref{MarkFibObj}, and observe that given a marked $\omega$-category $\cD$ and $a\in\eq_{\indind}\cD$ for $\indind>1$, then $a\in\eq_{\indind-1}\hom_{\cD}(d_0^-a,d_0^+a)$.
Further, the functor
\[\Sigma\colon\omegacat^+_{\mathrm{coind}}
\to(\omegacat^+_{\mathrm{coind}})_{*,*}\] sends the marked $\omega$-functor $i_n^+\colon\cC_n^{\flat}\hookrightarrow(\cC_{n+1},\{e_{n+1}\}\cup\id\cC_{n+1})$ to the marked $\omega$-functor $i_{n+1}^+\colon\cC_{n+1}^{\flat}\hookrightarrow(\cC_{n+2},\{e_{n+2}\}\cup\id\cC_{n+2})$. Hence,
by \cref{thm:model structure on marked omega categories}\ref{MarkFibFib}, the functor
\[\hom\colon(\omegacat^+_{\mathrm{coind}})_{*,*}\to \omegacat^+_{\mathrm{coind}}\]
preserves fibrations between fibrant objects. Finally, by definition of acyclic cofibrations and using the adjunction $\Sigma\dashv\hom$, the functor \[\Sigma\colon\omegacat^+_{\mathrm{coind}}
\to
\colon(\omegacat^+_{\mathrm{coind}})_{*,*}\]
preserves acyclic cofibrations, and so does the functor
\[U\Sigma\colon\omegacat^+_{\mathrm{coind}}\to 
(\omegacat^+_{\mathrm{coind}})_{*,*}
\to
\colon\omegacat^+_{\mathrm{coind}},\]
as desired.
\end{proof}

\subsection{The marked model for the coherent \omegat-equivalence}

\begin{const}
Let $\cB$, resp.\ $\cA$, denote the $\omega$-category freely generated by the following datum
\[\begin{tikzcd}
	& p \\
	q && q
	\arrow["{g'}", from=2-1, to=1-2]
	\arrow["f", from=1-2, to=2-3]
	\arrow[""{name=0, anchor=center, inner sep=0}, Rightarrow, no head, from=2-1, to=2-3]
	\arrow["\secondiso", shorten <=3pt, shorten >=3pt, Rightarrow, from=1-2, to=0]
\end{tikzcd}
\quad\text{resp.}\quad\begin{tikzcd}
	p && p \\
	& q
	\arrow[""{name=0, anchor=center, inner sep=0}, Rightarrow, no head, from=1-1, to=1-3]
	\arrow["f"', from=1-1, to=2-2]
	\arrow["{g}"', from=2-2, to=1-3]
	\arrow["\firstiso"', shorten <=3pt, shorten >=3pt, Rightarrow, from=2-2, to=0]
\end{tikzcd}\]
Let $(\cA,t\cA)$, resp.~$(\cB,t\cB)$,
denote the marked $\omega$-category for which $t\cA$, resp.~$t\cB$, is minimal with the property that $t\cA\supseteq\id\cA\cup\{f,\firstiso\}$, 
resp.~$t\cB\supseteq\id\cB\cup\{f,\secondiso\}$.
Let
$(\overline{\cQ},t\overline{\cQ})$ denote the marked $\omega$-category obtained as the pushout in $\omegacat^+$:
\[\begin{tikzcd}
	{\cC_1^{\sharp}} & {(\cB,t\cB)} \\
	{(\cA,t\cA)} & {(\overline{\cQ},t\overline{\cQ})}
	\arrow["{f}", from=1-1, to=1-2]
	\arrow["{f}"', from=1-1, to=2-1]
	\arrow[from=2-1, to=2-2]
	\arrow[from=1-2, to=2-2]
\end{tikzcd}\]
We refer the reader to \cite[Construction~2.14]{HenryLoubaton} for a description of pushouts in $\omegacat^+$.
\end{const}

\begin{lem}
\label{lemma:an important acyclic cofibration}
The marked $\omega$-functor
\[f\colon\cC_1^{\sharp}\to  (\overline{\cQ},t\overline{\cQ})\]
is an acyclic cofibration in $\omegacat^+_{\mathrm{coind}}$.
\end{lem}

\begin{proof}
The marked $\omega$-functors
\[f\colon\cC_1^{\sharp}\to(\cA,t\cA)\quad\text{ and }\quad f\colon\cC_1^{\sharp}\to(\cB,t\cB)\]
can be recognized as equation inclusions (in the sense of \cite[Definition~3.1]{HenryLoubaton}),
so they are by \cite[Corollary~3.24]{HenryLoubaton} acyclic cofibrations
in the inductive left semi-model structure from \cite[Corollary~2.38]{HenryLoubaton}, hence in the left semi-model structure $\omegacat^+_{\mathrm{coind}}$, which was constructed as a left Bousfield localization of it (cf.~\cref{thm:model structure on marked omega categories}).
Furthermore, since acyclic cofibrations are closed under pushouts,
the marked $\omega$-functor
\[(\cA,t\cA)\to (\overline{\cQ},t\overline{\cQ})\]
is also an acyclic cofibration in $\omegacat^+_{\mathrm{coind}}$, and hence so is the composite
\[f\colon\cC_1^{\sharp}\to (\cA,t\cA)\to (\overline{\cQ},t\overline{\cQ}),\]
as desired.
\end{proof}

\begin{const}
\label{TruncatedOverlineGuy}
Let $(\overline{\omega\cE}^{(0)},t\overline{\omega\cE}^{(0)}) \coloneqq \cC_1^\sharp$.
For $\indind > 0$, we define inductively $(\overline{\omega\cE}^{(\indind)},t\overline{\omega\cE}^{(\indind)})$ to be a marked $\omega$-category coming with a triple of marked $\omega$-functors
\[\overline\imath_{\indind}\colon  (\overline{\omega\cE}^{(\indind-1)},t\overline{\omega\cE}^{(\indind-1)}) \to (\overline{\omega\cE}^{(\indind)},t\overline{\omega\cE}^{(\indind)}),\]
\[\firstiso_{\indind}, \secondiso_{\indind}\colon \Sigma( \overline{\omega\cE}^{(\indind-1)},t\overline{\omega\cE}^{(\indind-1)}) \to (\overline{\omega\cE}^{(\indind)},t\overline{\omega\cE}^{(\indind)}).\]

\begin{itemize}[leftmargin=*]
\item For $\indind = 1$, we let $(\overline{\omega\cE}^{(1)},t\overline{\omega\cE}^{(1)}) \coloneqq (\overline{\cQ},t\overline{\cQ})$, we let $\overline\imath_1$ be the
marked $\omega$-functor
\[
f\colon\cC_1^\sharp\to(\overline{\cQ},t\overline{\cQ})
\]
and $\firstiso_1$, $\secondiso_1$ be defined by
\[
    \firstiso_1\colon \Sigma p \mapsto g \underset{0}* f, \; \Sigma q \mapsto \id_p, \;  \Sigma f\mapsto\firstiso\quad\text{ and }\quad
    \secondiso_1\colon \Sigma p \mapsto f \underset{0}*g', \; \Sigma q \mapsto \id_q, \;
    \Sigma f\mapsto\beta.
\]
\item For $\indind > 1$, we let $(\overline{\omega\cE}^{(\indind)},t\overline{\omega\cE}^{(\indind)})$, $\overline\imath_\indind$, $\firstiso_\indind$, and $\secondiso_\indind$ be defined by the pushout in $\omegacat^+$
\begin{equation}
    \label{omegacen_pushout2}
\begin{tikzcd}[column sep=scriptsize]
	{\Sigma(\overline{\omega\cE}^{(\indind-2)},t\overline{\omega\cE}^{(\indind-2)}) \amalg \Sigma(\overline{\omega\cE}^{(\indind-2)},t\overline{\omega\cE}^{(\indind-2)})} &&& {(\overline{\omega\cE}^{(\indind-1)},t\overline{\omega\cE}^{(\indind-1)})} \\
	{\Sigma(\overline{\omega\cE}^{(\indind-1)},t\overline{\omega\cE}^{(\indind-1)}) \amalg \Sigma( \overline{\omega\cE}^{(\indind-1)},t\overline{\omega\cE}^{(\indind-1)})} &&& {(\overline{\omega\cE}^{(\indind)},t\overline{\omega\cE}^{(\indind)})}.
	\arrow["{[\firstiso_{\indind-1}, \secondiso_{\indind-1}]}", from=1-1, to=1-4]
	\arrow["{\overline\imath_{\indind}}", from=1-4, to=2-4]
	\arrow["{\Sigma(\overline\imath_{\indind-1}) \amalg \Sigma (\overline\imath_{\indind-1})}", from=1-1, to=2-1,swap]
	\arrow["{[\firstiso_{\indind}, \secondiso_{\indind}]}", from=2-1, to=2-4,swap]
\end{tikzcd}
\end{equation}
\end{itemize}
\end{const}

\begin{lem}
\label{lemma:another important acyclic cofibration}
For all $k\geq0$ the marked $\omega$-functor
\[\overline\imath_{\indind}\colon(\overline{\omega\cE}^{(\indind-1)},t\overline{\omega\cE}^{(\indind-1)})\hookrightarrow(\overline{\omega\cE}^{(\indind)},t\overline{\omega\cE}^{(\indind)})\]
is an acyclic cofibration in $\omegacat^+_{\mathrm{coind}}$. In particular, $(\overline{\omega\cE}^{(\indind-1)},t\overline{\omega\cE}^{(\indind-1)})$ is cofibrant in $\omegacat^+_{\mathrm{coind}}$.
\end{lem}

\begin{proof}
One can deduce this by induction on $\indind\geq1$. The base case is \cref{lemma:an important acyclic cofibration}, and the inductive step is a consequence of the induction hypothesis and \eqref{omegacen_pushout2}.
\end{proof}

\begin{const}
\label{omegaEoverline}
We denote by $(\overline{\omega\cE},t\overline{\omega\cE})$
the colimit in $\omegacat^+$ given by
\[(\overline{\omega\cE},t\overline{\omega\cE})\coloneqq \colim[\ \dots\hookleftarrow (\overline{\omega\cE}^{(\indind)},t\overline{\omega\cE}^{(\indind)})\hookleftarrow\dots
\hookleftarrow(\overline{\omega\cE}^{(0)},t\overline{\omega\cE}^{(0)})\ ].\]
\end{const}

\begin{lem}
\label{lemma:another important acyclic cofibration2}
Given $k\geq0$, the marked $\omega$-functor
\[\overline\imath_{\indind,\infty}\colon(\overline{\omega\cE}^{(\indind)},t\overline{\omega\cE}^{(\indind)})\hookrightarrow(\overline{\omega\cE},t\overline{\omega\cE})\]
obtained as a structure map in the colimit cone from \cref{omegaEoverline}, is an acyclic cofibration in $\omegacat^+_{\mathrm{coind}}$.  In particular, $(\overline{\omega\cE},t\overline{\omega\cE})$ is cofibrant in $\omegacat^+_{\mathrm{coind}}$.
\end{lem}

\begin{proof}
This follows from \cref{lemma:another important acyclic cofibration}, the fact that the class of acyclic cofibrations is closed under transfinite composition, and the fact that acyclic cofibrations are cofibrations.
\end{proof}

We can understand the underlying $\omega$-category of $(\overline{\omega\cE},t\overline{\omega\cE})$:

\begin{lem}
\label{lem:comparaison widehat and overline 2}
Given $\indind\geq0$,
there exist $\omega$-functors
\[
 \eta^{(\indind)}\colon\widehat{\omega\mathcal{E}}^{(\indind)}\to \overline{\omega\mathcal{E}}^{(\indind)}\quad\text{ and }\quad
    \mu^{(\indind)}\colon \overline{\omega\mathcal{E}}^{(\indind)}\to \widehat{\omega\mathcal{E}}^{(\indind+1)}
\]
that make the following diagram in $\omega\cat$ commute:
\begin{equation}
\label{PropertyMuEta}
\begin{tikzcd}
	{\widehat{\omega\mathcal{E}}^{(\indind)}} && {\widehat{\omega\mathcal{E}}^{(\indind+1)}} && {\widehat{\omega\mathcal{E}}^{(\indind+2)}}. \\
	& { \overline{\omega\mathcal{E}}^{(\indind)}} && { \overline{\omega\mathcal{E}}^{(\indind+1)}}
	\arrow["{\eta^{(\indind)}}"{description}, from=1-1, to=2-2]
	\arrow["\imath_{\indind}", from=1-1, to=1-3,hook]
	\arrow["{\mu^{(\indind)}}"{description}, from=2-2, to=1-3]
	\arrow["{\eta^{(\indind+1)}}"{description}, from=1-3, to=2-4]
	\arrow["\overline\imath_{\indind}", from=2-2, to=2-4,hook]
	\arrow["{\mu^{(\indind+1)}}"{description}, from=2-4, to=1-5]
	\arrow["\imath_{\indind+1}", from=1-3, to=1-5,hook]
\end{tikzcd}
\end{equation}
\end{lem}

\begin{proof}
We construct the $\omega$-functors $\eta^{(\indind)}$ and $\mu^{(\indind)}$ by induction on $\indind\geq0$
For the base cases, we set $\eta^{(0)}$ and $\mu^{(0)}$ to be the $\omega$-functors
\[
\eta^{(0)}\colon\widehat{\omega\mathcal{E}}^{(0)}=
\partial \mathcal{C}_1\hookrightarrow \mathcal{C}_1=\overline{\omega\mathcal{E}}^{(0)}
\quad\text{ and }\quad
\mu^{(0)}\colon \overline{\omega\mathcal{E}}^{(0)}=\cC_1\xrightarrow{f_1}\cQ=\widehat{\omega\mathcal{E}}^{(1)},
\]
and we set $\eta^{(1)}$ and $\mu^{(1)}$ to be the unique $\omega$-functors
\[
\eta^{(1)}\colon\widehat{\omega\mathcal{E}}^{(1)}=
\cQ\hookrightarrow \overline{\cQ}=\overline{\omega\mathcal{E}}^{(1)}
\quad\text{ and }\quad
\mu^{(1)}\colon \overline{\omega\mathcal{E}}^{(1)}=\overline{\cQ}\to\widehat{\omega\mathcal{E}}^{(2)},
\]
which are identity on underlying $1$-categories and such that
\[
\mu^{(1)}\colon\alpha\mapsto \alpha_1(\Sigma f)\quad\text{ and }\quad\mu^{(1)}\colon\beta\mapsto \beta_1(\Sigma f).
\]
For the inductive step, we assume that $\eta^{(\indind)}$ and $\mu^{(\indind)}$ have been constructed, and we now construct $\eta^{(\indind+1)}$ and $\mu^{(\indind+1)}$.
Using \cref{SigmaAndU,MarkedAdjointOfSuspension} and \eqref{omegacen_pushout2}, we see that there is a commutative diagram in $\omega\cat$:
\[\begin{tikzcd}
	{\Sigma(\widehat{\omega \mathcal{E}}^{(\indind)})\amalg\Sigma(\widehat{\omega \mathcal{E}}^{(\indind)})} & {\Sigma(\widehat{\omega \mathcal{E}}^{(\indind-1)})\amalg\Sigma(\widehat{\omega \mathcal{E}}^{(\indind-1)})} & {\widehat{\omega \mathcal{E}}^{(\indind)}} \\
	{\Sigma(\overline{\omega \mathcal{E}}^{(\indind)})\amalg\Sigma(\overline{\omega \mathcal{E}}^{(\indind)})} & {\Sigma(\overline{\omega \mathcal{E}}^{(\indind-1)})\amalg\Sigma(\overline{\omega \mathcal{E}}^{(\indind-1)})} & {\overline{\omega \mathcal{E}}^{(\indind)}}
	\arrow[from=1-2, to=1-3]
	\arrow["{\eta^{(\indind)}}", from=1-3, to=2-3]
	\arrow["{\Sigma\eta^{(\indind-1)}\amalg \Sigma\eta^{(\indind-1)}}", from=1-2, to=2-2]
	\arrow["{\Sigma\eta^{(\indind)}\amalg \Sigma\eta^{(\indind)}}", from=1-1, to=2-1]
	\arrow[from=2-2, to=2-3]
	\arrow[from=2-2, to=2-1]
	\arrow[from=1-2, to=1-1]
\end{tikzcd}\]
and, using \eqref{omegacen_pushout2}, we define $\eta^{(\indind+1)}$ as the $\omega$-functor
\[\eta^{(\indind+1)}\colon \widehat{\omega\mathcal{E}}^{(\indind+1)}\to \overline{\omega\mathcal{E}}^{(\indind+1)}\]
induced at the level of colimits by this map of spans in $\omega\cat$. Similarly, using again \cref{SigmaAndU,MarkedAdjointOfSuspension} and \eqref{omegacen_pushout2}, we see that there is a commutative diagram in $\omega\cat$:
\[\begin{tikzcd}
	{\Sigma(\overline{\omega \mathcal{E}}^{(\indind)})\amalg\Sigma(\overline{\omega \mathcal{E}}^{(\indind)})} & {\Sigma(\overline{\omega \mathcal{E}}^{(\indind-1)})\amalg\Sigma(\overline{\omega \mathcal{E}}^{(\indind-1)})} & {\overline{\omega \mathcal{E}}^{(\indind)}} \\
	{\Sigma(\widehat{\omega \mathcal{E}}^{(\indind+1)})\amalg\Sigma(\widehat{\omega \mathcal{E}}^{(\indind+1)})} & {\Sigma(\widehat{\omega \mathcal{E}}^{(\indind)})\amalg\Sigma(\widehat{\omega \mathcal{E}}^{(\indind)})} & {\widehat{\omega \mathcal{E}}^{(\indind+1)}}
	\arrow[from=2-2, to=2-3]
	\arrow["{\mu^{(\indind)}}", from=1-3, to=2-3]
	\arrow["{\Sigma\mu^{(\indind-1)}\amalg \Sigma\mu^{(\indind-1)}}", from=1-2, to=2-2]
	\arrow["{\Sigma\mu^{(\indind)}\amalg \Sigma\mu^{(\indind)}}", from=1-1, to=2-1]
	\arrow[from=1-2, to=1-3]
	\arrow[from=1-2, to=1-1]
	\arrow[from=2-2, to=2-1]
\end{tikzcd}\]
and, using \eqref{omegacen_pushout2}, we define $\mu^{(\indind+1)}$ as the $\omega$-functor
\[\mu^{(\indind+1)}\colon   \overline{\omega\mathcal{E}}^{(\indind+1)}\to\widehat{\omega\mathcal{E}}^{(\indind+2)}\]
induced at the level of colimits by this map of spans in $\omega\cat$.
One can finally show, by induction on $\indind\geq0$,
that the $\omega$-functors $\eta^{(\indind)}$, $\mu^{(\indind)}$, $\eta^{(\indind+1)}$ and $\mu^{(\indind+1)}$ fit into the desired commutative diagram in $\omega\cat$.
\end{proof}

\begin{prop}
\label{underlying omega category of overline omega cE}
There is an isomorphism in $\omega\cat$

\[\mu\colon\overline{\omega\cE}=U(\overline{\omega\cE},t\overline{\omega\cE})\cong\widehat{\omega\cE}\colon\eta.\]
\end{prop}

\begin{proof}
From the property \eqref{PropertyMuEta}, one can deduce that the $\omega$-functors $\eta^{(\indind)}$ and $\mu^{(\indind)}$ from \cref{lem:comparaison widehat and overline 2} define by construction the components of two natural transformations with respect to $\indind\in\mathbb N$. By taking the $\omega$-functor induced at the level of colimits over $n\in\mathbb N$ we then obtain $\omega$-functors
\[
\colim_{\indind\in\mathbb N}
 \eta^{(\indind)}\colon  \colim_{\indind\in\mathbb N}\widehat{\omega\mathcal{E}}^{(\indind)}\to  \colim_{\indind\in\mathbb N}\overline{\omega\mathcal{E}}^{(\indind)}, \;\; 
     \colim_{\indind\in\mathbb N}\mu^{(\indind)}\colon  \colim_{\indind\in\mathbb N}\overline{\omega\mathcal{E}}^{(\indind)}\to  \colim_{\indind\in\mathbb N}\widehat{\omega\mathcal{E}}^{(\indind+1)},
\]
which can be identified with $\omega$-functors
\[
 \eta\colon\widehat{\omega\mathcal{E}}\to \overline{\omega\mathcal{E}}\quad\text{ and }\quad
    \mu\colon \overline{\omega\mathcal{E}}\to \widehat{\omega\mathcal{E}}.
\]
From the property \eqref{PropertyMuEta}, one can also deduce that $\mu$ and $\eta$ are inverse to each other, concluding the proof.
\end{proof}

\begin{lem} \label{lem:natural_and_sharp_are_the_same_for_both_biinvertibilities}
    The inverse isomorphisms $\mu$ and $\eta$ in $\omega\cat$
    induce inverse
    isomorphisms in $\omegacat^+$
    \[
        \mu\colon \overline{\omega\cE}^\natural = \overline{\omega\cE}^\sharp \cong \widehat{\omega\cE}^\sharp = \widehat{\omega\cE}^\natural :\! \eta.
    \]
\end{lem}

\begin{proof}

Since $(-)^\sharp$ is a functor we obtain inverse isomorphisms in $\omegacat^+$
   \[
        \mu\colon \overline{\omega\cE}^\sharp \cong \widehat{\omega\cE}^\sharp :\! \eta.
    \]
    By \cref{prop:all_cells_of_omegahat_are_equivalences,prop:invertible is the same as biinvertible}, all cells of $\widehat{\omega\cE}$ above dimension 0 are $\omega$-equivalences, which implies that \[\widehat{\omega\cE}^\sharp = \widehat{\omega\cE}^\natural.\]
    By \cref{underlying omega category of overline omega cE} we obtain that
\[\overline{\omega\cE}^\sharp = \overline{\omega\cE}^\natural.\]
    This concludes the proof.
\end{proof}

\begin{prop}
\label{lemma:first canonical weak equivalence}
The $\omega$-functor $\mu$ determines an acyclic cofibration in $\omegacat^+_{\mathrm{coind}}$
\[\mu\colon(\overline{\omega\cE},t\overline{\omega\cE})\hookrightarrow\overline{\omega\cE}^{\natural} \cong \widehat{\omega\cE}^{\natural}\]
\end{prop}
\begin{proof}
    The existence of the marked $\omega$-functor follows from Lemma \ref{lem:natural_and_sharp_are_the_same_for_both_biinvertibilities} and the adjunction $U \dashv (-)^\sharp$, and the fact that it is an acyclic cofibration follows from \cref{lemma:an other important acyclic cofibration,underlying omega category of overline omega cE}.
\end{proof}

\begin{lem}
\label{lemma:second canonical weak equivalence}
Given $\indind\geq0$, the marked $\omega$-functor
\[f_{\indind}\colon\cC_1^{\sharp}\to(\overline{\omega\cE}^{(\indind)},t\overline{\omega\cE}^{(\indind)})\]
is an acyclic cofibration in $\omegacat^+_{\mathrm{coind}}$. In particular, by two-out-of-three for weak equivalences in $\omegacat^+_{\mathrm{coind}}$, we obtain that the marked $\omega$-functor
\[\overline{\imath}_{\indind}\colon(\overline{\omega\cE}^{(\indind)},t\overline{\omega\cE}^{(\indind)})\to (\overline{\omega\cE}^{(\indind+1)},t\overline{\omega\cE}^{(\indind+1)})\]
is an acyclic cofibration in $\omegacat^+_{\mathrm{coind}}$.
\end{lem}

\begin{proof}
We prove this by induction on $\indind\geq1$. The base case $\indind=1$ is \cref{lemma:an important acyclic cofibration}, and we now show the induction step, assuming the statement to be true for $\indind-1$. We have that the marked $\omega$-functor
\[f_{\indind-1}\colon\cC_1^{\sharp}\to(\overline{\omega\cE}^{(\indind-1)},t\overline{\omega\cE}^{(\indind-1)})\]
is an acyclic cofibration in $\omegacat^+_{\mathrm{coind}}$. By \cref{SuspensionHomotopical},
we obtain that the marked $\omega$-functor
\[\Sigma\cC_1^{\sharp}\amalg \Sigma\cC_1^{\sharp}\to\Sigma(\overline{\omega\cE}^{(\indind-1)},t\overline{\omega\cE}^{(\indind-1)})\amalg\Sigma(\overline{\omega\cE}^{(\indind-1)},t\overline{\omega\cE}^{(\indind-1)})\]
is an acyclic cofibration in $\omegacat^+_{\mathrm{coind}}$.
By closure of the class of acylic cofibrations under pushouts, we obtain that the marked $\omega$-functor
\[(\overline{\cQ},t\overline{\cQ})\to(\overline{\omega\cE}^{(\indind)},t\overline{\omega\cE}^{(\indind)})\]
is an acyclic cofibration in $\omegacat^+_{\mathrm{coind}}$. By \cref{lemma:an important acyclic cofibration}, we obtain that the composite marked $\omega$-functor
\[f_{\indind}\colon\cC_1^\sharp\xrightarrow{f_1}(\overline{\cQ},t\overline{\cQ})\to(\overline{\omega\cE}^{(\indind)},t\overline{\omega\cE}^{(\indind)})\]
is an acyclic cofibration in $\omegacat^+_{\mathrm{coind}}$, as desired.
\end{proof}

\begin{prop}
\label{lemma:third canonical weak equivalence}
The unique marked $\omega$-functor
\[(\overline{\omega\cE},t\overline{\omega\cE})\to \cC_0^{\sharp}\]
is a weak equivalence in $\omegacat^+_{\mathrm{coind}}$.
\end{prop}

\begin{proof}
The marked $\omega$-functor
\[i_0^+\colon\cC_0^{\sharp}\hookrightarrow  \cC_1^{\sharp}\]
is by \cref{thm:model structure on marked omega categories} a weak equivalence in
$\omegacat^+_{\mathrm{coind}}$.
The marked $\omega$-functor
\[
f_1\colon\cC_1^{\sharp}\hookrightarrow(\overline{\cQ},t\overline{\cQ})
\]
is a weak equivalence in $\omegacat^+_{\mathrm{coind}}$ by \cref{lemma:an important acyclic cofibration}. The marked $\omega$-functor
\[(\overline{\cQ},t\overline{\cQ})\to(\overline{\omega\cE}^{(\indind)},t\overline{\omega\cE}^{(\indind)})\to(\overline{\omega\cE}^{(\indind+1)},t\overline{\omega\cE}^{(\indind+1)})\to\dots\to  (\overline{\omega\cE},t\overline{\omega\cE})\]
is a weak equivalence in $\omegacat^+_{\mathrm{coind}}$
by \cref{lemma:second canonical weak equivalence}, using the fact that acyclic cofibrations are closed under transfinite composition.
So the composite marked $\omega$-functor
\[\cC_0^{\sharp}\xhookrightarrow{i_0^+} \cC_1^{\sharp}\xhookrightarrow{f_1}(\overline{\cQ},t\overline{\cQ})\to  (\overline{\omega\cE},t\overline{\omega\cE})\]
is a weak equivalence in $\omegacat^+_{\mathrm{coind}}$. By two-out-of-three,
the unique $\omega$-functor
\[(\overline{\omega\cE},t\overline{\omega\cE})\to\cC_0^{\sharp}\]
is then also a weak equivalence in $\omegacat^+_{\mathrm{coind}}$, as desired.
\end{proof}

We can finally now prove the main theorem, namely that the unique morphism $\widehat{\omega\cE}\to \cC_0$
is a weak equivalence in $\omega\cat_{\mathrm{can}}$:

\begin{proof}[Proof of \cref{omegaEcontractible}]
Consider the commutative diagram in $\omegacat^+_{\mathrm{coind}}$
\[\begin{tikzcd}
(\overline{\omega\cE},t\overline{\omega\cE})\arrow[r]\arrow[rd]&\widehat{\omega\cE}^{\natural}\arrow[d]\\
&\cC_0^{\sharp}.
\end{tikzcd}\]
By \cref{lemma:first canonical weak equivalence,lemma:third canonical weak equivalence}, the top and the diagonal marked $\omega$-functors are weak equivalences in $\omegacat^+_{\mathrm{coind}}$.
By two-out-of-three, so is the right vertical marked $\omega$-functor
\[\widehat{\omega\cE}^{\natural}\to \cC_0^{\sharp}.\]
By \cref{thm:model structure on marked omega categories}\ref{MarkFibObj}, the marked $\omega$-categories $\widehat{\omega\cE}^{\natural}$ and $\cC_0^{\sharp}$ are fibrant in $\omegacat^+_{\mathrm{coind}}$. By \cref{thm:model structure on marked omega categories}\ref{MarkFibWE}, the forgetful functor $U\colon\omegacat^+_{\mathrm{coind}}\to\omega\cat_{\mathrm{can}}$ preserves weak equivalences between fibrant objects, so the unique $\omega$-functor
\[\widehat{\omega\cE}=U(\widehat{\omega\cE}^{\natural})\to U(\cC_0^{\natural})= U(\cC_0^{\sharp})=\cC_0\]
is a weak equivalence in $\omega\cat_{\mathrm{can}}$, as desired.
\end{proof}

\bibliographystyle{amsalpha}
\bibliography{ref}

\newcommand{\etalchar}[1]{$^{#1}$}
\providecommand{\bysame}{\leavevmode\hbox to3em{\hrulefill}\thinspace}
\providecommand{\MR}{\relax\ifhmode\unskip\space\fi MR }
\providecommand{\MRhref}[2]{%
  \href{http://www.ams.org/mathscinet-getitem?mr=#1}{#2}
}
\providecommand{\href}[2]{#2}
\begin{thebibliography}{ABG{\etalchar{+}}23}

\bibitem[ABG{\etalchar{+}}23]{PolyBook}
Dimitri Ara, Albert Burroni, Yves Guiraud, Philippe Malbos, François Métayer,
  and Samuel Mimram, \emph{Polygraphs: From rewriting to higher categories},
  \href{https://arxiv.org/abs/2312.00429v1}{arXiv:2312.00429v1}, 2023.

\bibitem[AL20]{AraLucas}
Dimitri Ara and Maxime Lucas, \emph{The folk model category structure on strict
  {$\omega$}-categories is monoidal}, Theory Appl. Categ. \textbf{35} (2020),
  Paper No. 21, 745--808. \MR{4105933}

\bibitem[AM20]{AraMaltsiniotisJoin}
Dimitri Ara and Georges Maltsiniotis, \emph{Joint et tranches pour les
  $\infty$-cat\'{e}gories strictes}, M\'{e}m. Soc. Math. Fr. (N.S.) (2020),
  no.~165, vi+213.

\bibitem[BW24]{BataninWhiteLoc}
Michael Batanin and David White, \emph{Left {B}ousfield localization without
  left properness}, J. Pure Appl. Algebra \textbf{228} (2024), no.~6, Paper No.
  107570, 23. \MR{4670515}

\bibitem[Che07]{ChengOmegaDuals}
Eugenia Cheng, \emph{An $\omega$-category with all duals is an
  $\omega$-groupoid}, Applied Categorical Structures \textbf{15} (2007),
  439--453.

\bibitem[cli22]{clingmanThesis}
tslil clingman, \emph{Towards the theory of proof-relevant categories}, Ph.D.
  thesis, Johns Hopkins University, 2022.

\bibitem[FHM23]{FHM}
Soichiro Fujii, Keisuke Hoshino, and Yuki Maehara, \emph{Weakly invertible
  cells in a weak $\omega $-category},
  \href{https://arxiv.org/abs/2303.14907v2}{arXiv:2303.14907v2}, 2023.

\bibitem[Gol23]{GoldthorpeWeakEnrichment}
Zach Goldthorpe, \emph{Homotopy theories of {$(\infty,\infty)$}-categories as
  universal fixed points with respect to weak enrichment}, Int. Math. Res. Not.
  IMRN (2023), no.~22, 19592--19640. \MR{4669810}

\bibitem[Gol24]{GoldthorpeSheaves}
Zach Goldthorpe, \emph{Sheaves of $(\infty, \infty)$-categories},
  \href{https://arxiv.org/abs/2403.06926v3}{arXiv:2403.06926v3}, 2024.

\bibitem[Gur12]{GurskiBieq}
Nick Gurski, \emph{Biequivalences in tricategories}, Theory Appl. Categ.
  \textbf{26} (2012), No. 14, 349--384. \MR{2972968}

\bibitem[Had20]{HadzihasanovicDiagrammatic}
Amar Hadzihasanovic, \emph{Diagrammatic sets and rewriting in weak higher
  categories}, \href{https://arxiv.org/abs/2007.14505v1}{arXiv:2007.14505v1},
  2020.

\bibitem[Hen20]{HenryWeakmodelstructure}
Simon Henry, \emph{Weak model categories in classical and constructive
  mathematics}, Theory Appl. Categ. \textbf{35} (2020), Paper No. 24, 875--958.
  \MR{4112763}

\bibitem[Hen23]{HenryCombinatorialAccessible}
\bysame, \emph{Combinatorial and accessible weak model categories}, J. Pure
  Appl. Algebra \textbf{227} (2023), no.~2, Paper No. 107191, 46. \MR{4460364}

\bibitem[Hir21]{HirschhornOvercategories}
Philip~S. Hirschhorn, \emph{Overcategories and undercategories of cofibrantly
  generated model categories}, J. Homotopy Relat. Struct. \textbf{16} (2021),
  no.~4, 753--768. \MR{4343079}

\bibitem[HL23]{HenryLoubaton}
Simon Henry and Félix Loubaton, \emph{An inductive model structure for strict
  $\infty$-categories},
  \href{https://arxiv.org/abs/2301.11424v1}{arXiv:2301.11424v1}, 2023.

\bibitem[HORR23]{HORR}
Philip Hackney, Viktoriya Ozornova, Emily Riehl, and Martina Rovelli, \emph{An
  {$(\infty,2)$}-categorical pasting theorem}, Trans. Amer. Math. Soc.
  \textbf{376} (2023), no.~1, 555--597. \MR{4510118}

\bibitem[Lac02]{lack1}
Stephen Lack, \emph{A {Q}uillen model structure for 2-categories}, $K$-Theory
  \textbf{26} (2002), no.~2, 171--205. \MR{1931220}

\bibitem[Lac04]{lack2}
\bysame, \emph{A {Q}uillen model structure for bicategories}, $K$-Theory
  \textbf{33} (2004), no.~3, 185--197. \MR{2138540}

\bibitem[LMW10]{LMW}
Yves Lafont, Fran{\c{c}}ois M{\'e}tayer, and Krzysztof Worytkiewicz, \emph{A
  folk model structure on omega-cat}, Advances in Mathematics \textbf{224}
  (2010), no.~3, 1183--1231.

\bibitem[Lou23]{LoubatonNerves}
F{\'e}lix Loubaton, \emph{Kan conditions on the nerves of
  {{\(\omega\)}}-categories}, Bull. Soc. Math. Fr. \textbf{151} (2023), no.~2,
  331--406 (French).

\bibitem[OR21]{ORNerves2Cat}
Viktoriya Ozornova and Martina Rovelli, \emph{Nerves of 2-categories and
  2-categorification of $(\infty,2)$-categories}, Advances in Mathematics
  \textbf{391} (2021), 107948.

\bibitem[OR23]{ORquillen}
\bysame, \emph{A {Q}uillen adjunction between globular and complicial
  approaches to {$(\infty,n)$}-categories}, Adv. Math. \textbf{421} (2023),
  Paper No. 108980, 57. \MR{4571188}

\bibitem[OR24]{ORSurvey}
\bysame, \emph{What is an equivalence in a higher category?}, Bulletin of the
  London Mathematical Society \textbf{56} (2024), no.~1, 1--58.

\bibitem[Rez10]{rezkTheta}
Charles Rezk, \emph{A {C}artesian presentation of weak {$n$}-categories}, Geom.
  Topol. \textbf{14} (2010), no.~1, 521--571. \MR{2578310}

\bibitem[Ric20]{RiceCoinductive}
Alex Rice, \emph{Coinductive invertibility in higher categories},
  \href{https://arxiv.org/abs/2008.10307v2}{arXiv:2008.10307v2}, 2020.

\bibitem[RV16]{RiehlVerityMonads}
Emily Riehl and Dominic Verity, \emph{Homotopy coherent adjunctions and the
  formal theory of monads}, Adv. Math. \textbf{286} (2016), 802--888.
  \MR{3415698}

\bibitem[Ste04]{SteinerOmegaCatChain}
Richard Steiner, \emph{Omega-categories and chain complexes}, Homology Homotopy
  Appl. \textbf{6} (2004), no.~1, 175--200. \MR{2061574}

\bibitem[Str87]{StreetOrientedSimplexes}
Ross Street, \emph{The algebra of oriented simplexes}, J. Pure Appl. Algebra
  \textbf{49} (1987), no.~3, 283--335. \MR{920944}

\bibitem[Ver08]{VerityComplicialI}
Dominic Verity, \emph{Weak complicial sets. {I}. {B}asic homotopy theory}, Adv.
  Math. \textbf{219} (2008), no.~4, 1081--1149. \MR{2450607}

\end{thebibliography}

\end{document}